\documentclass[11pt]{amsart}
\usepackage[utf8]{inputenc}
\usepackage{amsmath,amssymb,amsthm}
\usepackage[colorlinks,citecolor=blue]{hyperref}
\usepackage{cleveref}
\usepackage{enumitem}
\usepackage{booktabs}
\usepackage{mathrsfs}
\usepackage{graphicx}
\usepackage{tikz}
\usetikzlibrary{patterns, shapes, arrows}
\usetikzlibrary{trees}
\usetikzlibrary{cd, graphs.standard}
\newcommand{\ZZ}{\mathbb{Z}}
\newcommand{\NN}{\mathbb{N}}

\newcommand{\QQ}{\mathbb{Q}}
\newcommand{\RR}{\mathbb{R}}

\newcommand{\TA}{\mathbb{A}}
\newcommand{\TB}{\mathbb{B}}

\newcommand{\aA}{{\mathcal A}}

\newcommand{\gG}{{\mathcal G}}
\newcommand{\iI}{{\mathcal I}}
\newcommand{\jJ}{{\mathcal J}}
\newcommand{\pP}{{\mathcal P}}
\newcommand{\qQ}{{\mathcal Q}}
\newcommand{\rR}{{\mathcal R}}
\newcommand{\uU}{{\mathcal U}}
\newcommand{\zZ}{{\mathcal Z}}
\newcommand{\wW}{{\mathcal W}}
\newcommand{\asc}{{\rm asc}}
\newcommand{\des}{{\rm des}}
\newcommand{\Ehr}{{\rm Ehr}}
\newcommand{\rev}{{\rm rev}}

\newcommand{\bDelta}{\boldsymbol\Delta}

\newtheorem{theorem}{Theorem}[section] 
\newtheorem{proposition}[theorem]{Proposition} 
\newtheorem{conjecture}[theorem]{Conjecture} 
 
\newtheorem{lemma}[theorem]{Lemma} 

\newtheorem{question}[theorem]{Question}
 
\newtheorem{remark}[theorem]{Remark} 
\newtheorem{example}[theorem]{Example}

\newcommand{\oeis}[1]{\href{https://oeis.org/#1}{{#1}}}

\newcommand{\down}{\mathscr{D}}
\newcommand{\up}{\mathscr{U}}

\newcommand{\cov}{\operatorname{cov}}

\newcommand{\Max}{\operatorname{Max}}

\newcommand{\re}{\operatorname{Re}}
\newcommand{\nara}{\operatorname{Nar}}
\newcommand{\biE}{\mathbb{E}}

\definecolor{racines}{RGB}{197,225,165}
\begin{document}
\title[Polytopes and posets associated to preorders]
{Polytopes and posets associated to preorders}

\author{Christos~A.~Athanasiadis}
\address{Department of Mathematics\\
National and Kapodistrian University of Athens\\
Panepistimioupolis\\ 15784 Athens, Greece}
\email{caath@math.uoa.gr}

\author{Frédéric~Chapoton}
\address{Institut de Recherche Mathématique 
Avancée, UMR 7501
\\ Université de Strasbourg et CNRS\\
7 rue René-Descartes\\
67000 Strasbourg, France}
\email{chapoton@unistra.fr}

\date{\today}
\thanks{ \textit{Mathematics Subject
Classifications}: 52B20, 05A15, 06A07}
\thanks{ \textit{Key words and phrases}.
Preorder, lattice polytope, generalized 
permutohedron, Ehrhart polynomial,
$h^\ast$-polynomial, zeta polynomial, reflexive polytope.
}

\begin{abstract}
Preorder polytopes, defined from preorders on 
finite sets, are introduced and studied from a 
lattice point enumeration point of view. They 
naturally generalize arbor polytopes, recently 
introduced and studied by the second named 
author. Preorder 
polytopes are shown to be lattice polytopes 
which satisfy a certain duality relating their 
Ehrhart polynomials with the zeta polynomials 
of their posets of lattice points. A combinatorial 
interpretation of the normalized volume of a 
preorder polytope is proven, together with 
formulas for the Ehrhart polynomial and 
the $h^\ast$-polynomial, and a combinatorial 
interpretation of the latter is conjectured.
Several conjectures and results on the lattice 
point enumeration of arbor polytopes are 
generalized to preorder polytopes, new 
conjectures are proposed and new interesting 
examples of preorder polytopes are studied. 
\end{abstract}

\maketitle

\section{Introduction}
\label{sec:intro}

Arbor polytopes are lattice polytopes
introduced and studied in~\cite{Cha25b}. This 
paper studies a natural generalization. Let us 
first recall the definition and main properties 
of arbor polytopes. Given a finite set $E$, an 
\emph{arbor} on $E$ is a rooted tree with 
vertices the blocks of a partition of $E$. The
\emph{arbor polytope} associated to such an arbor 
$\sigma$ is the polytope $\qQ_\sigma$ in $\RR^E$ 
defined by the linear inequalities $x_e \ge 0$
for $e \in E$ and
\begin{equation} \label{eq:arbor}
\sum_{e \in \down(v)} x_e \le |\down(v)|
\end{equation}
for every vertex $v$ of $\sigma$, where 
$\down(v)$ is the union of all vertices which 
are descendants of $v$ in $\sigma$ (including 
$v$ itself). The set of lattice points of 
$\qQ_\sigma$, denoted by $P_\sigma$ and 
partially ordered as a subposet of the product 
order $\NN^E$, also plays a central role in 
\cite{Cha25b}. It was shown there that 
$\qQ_\sigma$ is a simple lattice polytope with 
very interesting properties from a lattice 
point enumeration point of view. 

Let us present two conjectures and one result 
about arbor polytopes which partly motivated 
the present work (see Section~\ref{sec:pre}
for relevant background and any undefined 
terminology). Let us denote by $h_i(\sigma)$ 
the number of lattice points of $\qQ_\sigma$ 
having $i$ nonzero coordinates for $0 \le i 
\le n$ and set $h(\sigma,t) = \sum_{i=0}^n h_i
(\sigma)t^i$, where $n$ is the cardinality of 
the ground set $E$. Following \cite{Cha25b}, 
we call $h(\sigma,t)$ the \emph{$h$-polynomial} 
of $\sigma$.
\begin{conjecture} \label{conj:arbor-h}
{\rm (\cite[Conjecture~1.1]{Cha25b})} 
The polynomial $h(\sigma,t)$ is equal to the 
$h$-polynomial of an $n$-dimensional simplicial 
polytope for every arbor $\sigma$ on an 
$n$-element set $E$. In particular, $h(\sigma,t)$ 
is palindromic and unimodal.
\end{conjecture}
\begin{conjecture}
{\rm (\cite[Conjecture~3.2]{Cha25b})}
\label{conj:arbor-roots}
For every arbor $\sigma$, all roots of the
Ehrhart polynomial of $\qQ_\sigma$ are real and 
lie in the interval $[-1,0]$. In particular, the
$h^\ast$-polynomial of $\qQ_\sigma$ has only real
roots.
\end{conjecture}

An arbor $\sigma$ is said to be \emph{linear} if 
the associated rooted tree is a path. The reverse 
arbor $\rev(\sigma)$ is then defined by reversing 
the order of the vertices of this path. The 
following statement was conjectured in 
\cite[Section~4]{Cha25b}. 
\begin{theorem} \label{thm:Ath26}
{\rm (\cite[Theorem~1.3]{Ath26})}
The Ehrhart polynomial of $\qQ_\sigma$ and the
zeta polynomial of $P_{\rev(\sigma)}$ are related 
by the identity 
\[ \Ehr(\qQ_\sigma, t) = 
   \zZ(P_{\rev(\sigma)}, t+1) \]
for every linear arbor $\sigma$.
\end{theorem}

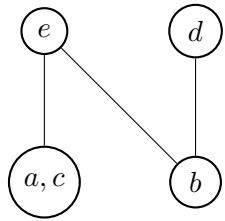
\begin{figure}[!h]
\centering
\begin{tikzpicture}
\tikzstyle{every node}=[draw,shape=circle, thick]

\node at (0,0) (ac) {$a,c$};
\node at (2,0) (b) {$b$};
\node at (0,2) (e) {$e$};
\node at (2,2) (d) {$d$};

\path [draw] (ac) -- (e);
\path [draw] (b) -- (e);
\path [draw] (b) -- (d);
\end{tikzpicture}
\caption{The Hasse diagram of a preorder on the 
         set $\{a,b,c,d,e\}$.}
\label{fig1}
\end{figure}

This paper generalizes arbor polytopes by 
replacing arbors on finite sets by arbitrary 
preorders. A \emph{preorder} on a finite set 
$E$ can be viewed as a partial order on the 
set of blocks of a partition of $E$, called 
vertices (see Section~\ref{sec:preorders}). 
The example shown in Figure~\ref{fig1} will 
serve as a running example. Given a preorder 
$\tau$ on a finite set $E$, the \emph{preorder 
polytope} associated to $\tau$ is the polytope 
$\qQ_\tau$ in $\RR^E$ defined by the linear 
inequalities $x_e \ge 0$ for $e \in E$ and
\begin{equation} \label{eq:Q-tau-ineq-b}
\sum_{e \in \iI} x_e \le |\iI|
\end{equation}
for all order ideals $\iI$ of $\tau$. This 
definition recovers and extends the definition 
of arbor polytopes. Indeed, every arbor $\sigma$ 
can be considered as a preorder $\tau$ having 
the root of $\sigma$ as the maximum vertex. 
Then, for every vertex $v$ of $\sigma$ the set 
$\down(v)$ is a principal order ideal of $\tau$ 
and inequalities of the form 
\eqref{eq:Q-tau-ineq-b} associated to 
nonprincipal order ideals of $\tau$ are 
consequences (namely sums) of inequalities 
associated to principal order ideals. 

It is not obvious that preorder polytopes are
lattice polytopes and no longer true that they 
are always simple. Our first main result shows
that preorder polytopes are indeed lattice 
polytopes which form a subclass of the class of 
generalized permutohedra \cite{post09}.
\begin{theorem} \label{thm:main-a}
Let $\tau$ be a preorder on an $n$-element set. 
Then, $\qQ_\tau$ is an $n$-dimensional lattice 
polytope which is unimodularly equivalent to a 
generalized permutohedron.   
\end{theorem}

An advantage of considering preorder polytopes,
compared to the more restrictive class of arbor 
polytopes, is that there is a natural notion of
duality for preorders. This notion allows for a 
generalization of Theorem~\ref{thm:Ath26} which
we call \emph{Ehrhart-Zeta duality}. Another 
obvious advantage is that the new setting allows 
for more interesting examples. We denote by 
$\tau^\ast$ the dual of the preorder $\tau$ and
define the poset $P_\tau$ of lattice points of
$\qQ_\tau$ as we did for arbors. The following 
statement is the second main result of this 
paper.  
\begin{theorem} \label{thm:main-b}
Let $\tau$ be a preorder on an $n$-element set 
$E$.
\begin{itemize}
\itemsep=0pt
\item[{\rm (a)}]
The Ehrhart polynomial of $\qQ_\tau$ is given by 
the formula 
\[ \Ehr(\qQ_\tau, t) = \sum_{(a_e) \in 
   P_{\tau^\ast}} 
   \prod_{e \in E} \binom{t+a_e-1}{a_e}, \]
where the sum ranges over all lattice points 
$(a_e)$ of $\qQ_{\tau^\ast}$. In particular, the 
number of lattice points of $\qQ_\tau$ is equal to
that of $\qQ_{\tau^\ast}$. 

The normalized volume of $\qQ_\tau$ is equal to the 
number of words $w \in E^n$ such that for every order 
ideal $\iI$ of $\tau$, the elements of $\iI$ appear a 
total of at least $|\iI|$ times in $w$.

\item[{\rm (b)}]
The Ehrhart polynomial of $\qQ_\tau$ and the zeta 
polynomial of $P_{\tau^\ast}$ are related by the identity 
\[ \Ehr(\qQ_\tau, t) = \zZ(P_{\tau^\ast}, t+1). \]
\end{itemize}

\end{theorem}

The proofs of both results rely on the theory 
of generalized permutohedra \cite{post09}. An 
interesting feature of our study of preorder 
polytopes is the variety of conjectures on their 
lattice point enumeration that can be stated, 
some of which extend 
Conjecture~\ref{conj:arbor-h} and (partly) 
Conjecture~\ref{conj:arbor-roots} on arbor 
polytopes. 

The content and structure of this paper can be 
summarized as follows. Section~\ref{sec:pre} 
discusses notation and reviews basic 
definitions and background on poset and preposet
combinatorics, Ehrhart theory of lattice 
polytopes, generalized permutohedra and related 
topics. Section~\ref{sec:basic} establishes
basic properties of preorder polytopes and 
their posets of lattice points. 
Theorems~\ref{thm:main-a} and~\ref{thm:main-b}
are proven in Sections~\ref{sec:basic} 
and~\ref{sec:Ehrhart}, respectively, by 
applying results from the theory of generalized 
permutohedra \cite{post09}. 
Section~\ref{sec:Ehrhart} also conjectures
a stronger form of Ehrhart-Zeta duality, as well 
as a combinatorial interpretation of the 
$h^\ast$-polynomial of any preorder polytope
which refines that of the normalized volume
given by Theorem~\ref{thm:main-b} (see 
Conjectures~\ref{conj:matrixEZ} 
and~\ref{conj:h-ast}) and provides partial 
results and a formula for the 
$h^\ast$-polynomial. Ehrhart polynomials of 
preorder polytopes may have nonreal complex 
roots. However, a weaker property (called 
magic positivity) which implies the 
real-rootedness of the $h^\ast$-polynomial 
is conjectured in Section~\ref{sec:Ehrhart} 
(see Conjecture~\ref{conj:magic}). 

Section~\ref{sec:h-vectors} discusses the 
natural extension of 
Conjecture~\ref{conj:arbor-h} for preorder 
polytopes and conjectures the 
$\gamma$-positivity and real-rootedness of
$h$-polynomials of preorders, as well as 
their invariance under duality. 
Section~\ref{sec:reflexive} studies reflexive 
polytopes which can be obtained as 
Minkowski sums of preorder polytopes and 
standard simplices. Section~\ref{sec:double}
provides a formula for the double Ehrhart
polynomial of a preorder polytope and the 
standard simplex and conjectures an identity
(see Conjecture~\ref{conj:double}) which is
stronger than Conjecture~\ref{conj:h-ast}. 
Section~\ref{sec:M-triangles} briefly 
discusses an important enumerative invariant 
of the poset of lattice points of a preorder 
polytope, namely the $M$-triangle. 
Section~\ref{sec:exa} discusses some 
interesting examples and confirms some of 
the conjectures of this paper for them.

\section{Preliminaries}
\label{sec:pre}

This section establishes notation to be followed 
throughout the paper and reviews definitions and
background on the combinatorics of partial orders
and preorders, lattice polytopes and generalized
permutohedra. More information about these topics
can be found in \cite{HiAC,StaCCA,StaEC1,post09}.

\subsection{Notation}
\label{sec:notation}

Let $E$ be a nonempty finite set. We will denote 
by ${({\rm \mathbf{e}}_e)}_{e \in E}$ the standard
basis of the space $\RR^E$. We will use boldface 
letters to denote vectors $\mathbf{x} = 
(x_e)_{e \in E}$ of $\RR^E$, or collections of 
real numbers. We will denote by $|\mathbf{x}|$ 
the sum of the coordinates of a point $\mathbf{x}
\in \RR^E$ and by $|S|$ or $\# S$ the cardinality 
of a finite set $S$.

\subsection{Preorders}
\label{sec:preorders}

Let $E$ be a nonempty finite set. A \emph{preorder} 
on $E$ is defined as a reflexive and transitive 
relation $\le$ on $E$. Declaring two elements $a, b
\in E$ to be equivalent if $a \le b$ and $b \le a$, 
one gets an equivalence relation on $E$. The induced 
order on the set of equivalence classes, which we  
call \emph{vertices}, is a partial order denoted by 
$\preceq$. Thus, a preorder $\le$ on the ground set 
$E$ can also be defined as a partial order $\preceq$ 
on the set of blocks of a partition of $E$. We will 
depict $\le$ by the Hasse diagram 
\cite[Section~3.1]{StaEC1} of the associated partial 
order $\preceq$ (where, as usual, the order increases 
from bottom to top). In the language of 
species~\cite{BLL98}, we are working with the 
composition of the species of partial orders with 
that of nonempty finite sets. Preorders on a finite 
set $E$ are in one-to-one correspondence with 
topologies on $E$ \cite[Exercise~3.3]{StaEC1}. We 
will usually denote preordered finite sets by small 
greek letters. We will refer to the cardinality of 
the ground set $E$ as the \emph{size} of a preorder 
on $E$ and to the cardinality of a vertex as the 
\emph{size} of this vertex. Figure~\ref{fig1} shows 
the Hasse diagram of a preorder of size 5; there is 
one vertex of size two and three single element 
vertices.

Basic constructions for posets 
\cite[Sections~3.1-3.2]{StaEC1}, such as 
disjoint unions, ordinal sums, duality, order 
ideals and filters and principal order ideals and 
filters, extend naturally to preorders. For example, 
an \emph{order ideal} of a preorder on the ground 
set $E$ is defined as any subset $\iI \subseteq E$ 
such that for every $b \in \iI$ we have $a \in \iI$ 
for every $a \in E$ with $a \le b$. Thus, order 
ideals are unions of vertices which are ideals of 
the associated partial order $\preceq$. Similarly, 
an \emph{order filter} (or \emph{dual order ideal}) 
is defined as any subset $\jJ \subseteq E$ such 
that for every $a \in \jJ$ we have $b \in \jJ$ for 
every $b \in E$ with $a \le b$. We will denote by
$\down(e)$ (respectively, $\up(e)$) the principal
order ideal (respectively, principal order filter)
generated by $e \in E$. We will sometimes abuse notation
and write $\down(v)$ and $\up(v)$ for a vertex $v$.

A preorder $\le$ on $E$ is called
a \emph{chain} if $a \le b$ or $b \le a$ for all 
$a, b \in E$ (equivalently, the induced partial 
order $\preceq$ is a total order). Chains are 
preorders associated to linear arbors 
\cite[Section~4]{Cha25b}; the corresponding 
arbor polytopes were studied in~\cite{Ath26} 
and called \emph{composition polytopes}.

\subsection{Poset combinatorics}
\label{sec:poset-comb}

Let $(P, \preceq)$ be a finite poset and let 
$d$ be the length (one less than the cardinality) 
of the longest chain in $P$. The \emph{zeta 
polynomial} \cite[Section~3.12]{StaEC1} of $P$ 
is the unique polynomial $\zZ(P, t)$ 
with the property that $\zZ(P, m+1)$ is equal 
to the number of $m$-element multichains in $P$, 
meaning sequences $(p_1, p_2,\dots,p_m)$ of 
elements of $P$ such that $p_1 \preceq p_2 
\preceq \cdots \preceq p_m$, for every $m \in 
\NN$. The degree of $\zZ(P, t)$ is equal to $d$
and the leading coefficient is equal to the 
number of (necessarily maximal) chains of $P$ 
of length $d$, divided by $d!$. We say that 
$P$ is \emph{graded} if all maximal chains in
$P$ have length $d$.

\subsection{Ehrhart theory}
\label{sec:Ehr-hstar}

Let $\qQ$ be a $d$-dimensional lattice polytope 
in $\RR^n$ (thus, every vertex of $\qQ$ belongs 
to $\ZZ^n$). The \emph{Ehrhart polynomial} of 
$\qQ$ is the unique polynomial $\Ehr(\qQ, t)$ 
with the property that $\Ehr(\qQ, m) = \# 
(m\qQ \cap \NN^n)$ for every $m \in \NN$. The 
Ehrhart polynomial $\Ehr(\qQ, t)$ has degree
$d$ and thus
\begin{equation} \label{eq:def-hstar} 
\sum_{m \ge 0} \Ehr(\qQ, m) t^m = 
\frac{h^\ast(\qQ, t)}{(1-t)^{d+1}}
\end{equation}
for some polynomial $h^\ast(\qQ, t)$ of degree at 
most $d$, called the \emph{$h^\ast$-polynomial}. 
The coefficients of $h^\ast(\qQ, t)$ are in fact 
nonnegative integers \cite[Theorem~33.9]{HiAC} 
\cite[Section~I.5]{StaCCA}. The sequence 
$h^\ast(\qQ) = (h^\ast_0(\qQ), 
h^\ast_1(\qQ),\dots,h^\ast_d(\qQ))$ of 
coefficients of $h^\ast(\qQ, t) = \sum_{i=0}^d 
h^\ast_i(\qQ) t^i$ is the \emph{$h^\ast$-vector} 
of $\qQ$. The leading coefficient of $\Ehr (\qQ, 
t)$ is equal to the $d$-dimensional volume of 
$\qQ$. Multiplied by $d!$, this number is a 
positive integer called the \emph{normalized 
volume}; it is equal to the sum of the entries 
of the $h^\ast$-vector. Thus, the normalized 
volume is refined by the $h^\ast$-vector. 

The polytope $\qQ$ is said to be \emph{Ehrhart 
positive} if all coefficients of $\Ehr(\qQ, t)$
are nonnegative. Finding classes of Ehrhart 
positive lattice polytopes is a very 
interesting problem \cite{FH24, Liu19}.
Following~\cite{FH24}, we say that $\Ehr(\qQ, 
t)$ is \emph{magic positive} if 
\[ \Ehr(\qQ, t) = \sum_{i=0}^d c_i 
   t^i(1+t)^{d-i} \]
for some nonnegative real numbers $c_0,
c_1,\dots,c_d$. By a theorem of Br\"and\'en 
\cite[Theorem~4.2]{Bra06}, this property (which 
is stronger than Ehrhart positivity) implies 
that $h^\ast(\qQ, t)$ has only real roots. 

\subsection{Generalized permutohedra}
\label{sec:genperm}

We briefly recall, with slightly adapted notation, 
the definitions and some basic results from 
Postnikov's theory of generalized permutohedra, as 
introduced in~\cite{post09}. This involves two
distinct notions, namely generalized permutohedra 
and the subclass of $Y$-type generalized 
permutohedra.

Let $E$ be a nonempty finite set. For a nonempty 
subset $I$ of $E$, the standard simplex $\Delta_I$ 
is the convex hull of the set of standard basis 
vectors 
$\{ {\rm \mathbf{e}}_e : e \in I\}$. $Y$-type 
generalized permutohedra are defined as Minkowski 
sums of scaled standard simplices. More precisely, 
given a collection $\mathbf{y} = {\{y_I\}}_I$ of 
nonnegative real numbers indexed by nonempty
subsets of $E$, 
\begin{equation} \label{eq:Py}
\pP^y(\mathbf{y}) = \sum_{\varnothing \not= 
I \subseteq E} y_I \Delta_I
\end{equation}
is the \emph{$Y$-type generalized permutohedron}
associated to $\mathbf{y}$.

On the other hand, generalized permutohedra are 
defined by linear equations and inequalities as 
follows. Given a collection $\mathbf{z} = 
{\{z_I\}}_I$ of nonnegative real numbers indexed 
by nonempty subsets of $E$,
\begin{equation} \label{eq:Pz}
  \pP^z(\mathbf{z}) = \bigg\{(x_e) \in \RR^E 
	\biggm| \sum_{e \in E} x_e = z_{E}, \quad 
	\sum_{i \in I} x_i \ge z_I \quad \forall I \not= 
	\varnothing \bigg\}
\end{equation}
is the \emph{generalized permutohedron} associated 
to $\mathbf{z}$. By~\cite[Proposition~6.3]{post09}, 
every $Y$-type generalized permutohedron is a 
generalized permutohedron. More precisely, 
defining the collection $\mathbf{z}$ from the 
collection $\mathbf{y}$ by setting
\begin{equation} \label{eq:zy}
z_J = \sum_{\varnothing \not= I \subseteq J} y_I
\end{equation}
for nonempty $J \subseteq E$, there is an equality 
of polytopes: $\pP^y(\mathbf{y}) = \pP^z
(\mathbf{z})$.

We now present a basic formula from~\cite{post09} 
for the Ehrhart polynomial of an integral $Y$-type 
generalized permutohedron. Let us consider nonempty 
subsets $I_0, I_1,\dots,I_m$ of $E$ and numbers 
$y_0, y_1,\dots,y_m \in \NN$. It is convenient to 
consider the bipartite graph $G$ with $m+1$ 
vertices $\ell_0, \ell_1,\dots,\ell_m$ on the left 
side and vertices $r_e$ for $e \in E$ on the right 
side for which $\ell_j$ is adjacent to $r_e$ if 
and only if $e \in I_j$. A sequence 
$(a_0, a_1,\dots,a_m) \in \NN^{m+1}$ is called 
\emph{$G$-draconian} \cite[Definition~9.2]{post09} 
if $a_0 + a_1 + \cdots + a_m = |E| - 1$ and 
\begin{equation} \label{eq:draco2}
1 + \sum_{\alpha \in \aA} a_\alpha \le \# 
	  \bigcup_{\alpha \in \aA} I_\alpha
\end{equation}
for every nonempty $\aA \subseteq \{0, 1,\dots,m\}$. 
The following statement is a consequence of 
\cite[Theorem~11.3]{post09}.
\begin{theorem} [\cite{post09}] \label{thm:post}
Assume that the bipartite graph $G$ just 
described has no isolated vertices and that $I_0 = 
E$. Then, for all nonnegative integers $y_j$, the 
Ehrhart polynomial of the $Y$-type generalized 
permutohedron
\[ \qQ = \sum_{j=0}^m y_j \Delta_{I_j} \]
in $\RR_{\ge 0}^E$ is given, by the formula 
\[ \Ehr(\qQ, t) = \sum_{(a_0, a_1,\dots,a_m) \in 
   \gG_G} \binom{y_0 t + a_0}{a_0}
   \prod_{i=1}^m \binom{y_i t + a_i - 1}{a_i}, 
	 \]
where $\gG_G$ stands for the set of $G$-draconian 
sequences. In particular, $\qQ$ is Ehrhart 
positive.
\end{theorem}

\section{Basic properties}
\label{sec:basic}

This section establishes basic properties of 
preorder polytopes and their posets of lattice 
points and proves Theorem~\ref{thm:main-a}. 
Throughout this and the following sections 
(unless stated otherwise), $\tau$ will be a 
preorder of size $n$ on the ground set $E$; we 
assume that $E$ does not contain 0. 

\subsection{Elementary properties}
\label{sec:elementary}

We recall that $\qQ_\tau$ is defined by $x_e 
\ge 0$ for every $e \in E$ and the linear 
inequalities~(\ref{eq:Q-tau-ineq-b}) for all 
order ideals $\iI$ of $\tau$, where $(x_e)$ are 
the standard coordinates of $\RR^E$. For our 
running example from Figure~\ref{fig1}, these
inequalities are 
\begin{align*} 
x_a + x_c & \le 2 \\ 
x_b & \le 1 \\ 
x_a + x_b + x_c + x_e & \le 4 \\
x_b + x_d & \le 2 \\
x_a + x_b + x_c + x_d + x_e & \le 5.
\end{align*}
Since $E$ is an order ideal of $\tau$, one of 
the linear inequalities defining $\qQ_\tau$ is 
\[ \sum_{e \in E} x_e \le n \]
and hence $\qQ_\tau$ is indeed a bounded 
polyhedron in $\RR^E$. Clearly, one may disregard 
the inequalities~(\ref{eq:Q-tau-ineq-b}) 
associated to order ideals which are disconnected
as subgraphs of the Hasse diagram of $\tau$,
since these are implied by the ones associated 
to connected order ideals.

\begin{remark} \label{rem:not-simple} \rm
(a) Contrary to the situation with arbor 
polytopes~\cite{Cha25b}, preorder polytopes are
not necessarily simple. This property fails
already for the three-element poset $\{a,b,c\}$ 
with relations $a \prec b$ and $a \prec c$. The 
preorder polytope for Figure~\ref{fig1} is not 
simple either; its $f$-vector is $(1, 27, 69, 72, 
38, 10, 1)$.

(b) The polytope $Q_\tau$ is an anti-blocking polytope, 
in the sense of~\cite{fulkerson}; see also 
\cite[Section~9.3]{schrijver}. This means that 
$Q_\tau \subset \RR^E_{\ge 0}$ and $\mathbf{y} 
\in \qQ_\tau \Rightarrow \mathbf{x} \in 
\qQ_\tau$ whenever $\mathbf{x}, \mathbf{y} \in 
\RR^E_{\ge 0}$ with $\mathbf{x} \le \mathbf{y}$ 
coordinatewise.
\qed
\end{remark}

\subsection{Preorder polytopes as generalized 
permutohedra}
\label{sec:Minkowski}

This section shows that every preorder polytope 
is unimodularly equivalent to a generalized 
permutohedron and deduces Theorem~\ref{thm:main-a}. 
This construction, which is explicit, will be 
essential for other results of this paper as well. 
We will work with a larger family of polytopes 
than that of preorder polytopes, since this will 
be convenient in Sections~\ref{sec:reflexive} and~\ref{sec:double}.

For $I \subseteq E$, let us denote by $\bDelta_I$ 
the lattice simplex in $\RR^E$ with vertices the 
basis vectors ${\rm \mathbf{e}}_e$ 
for $e \in I$ and the zero vector of $\RR^E$. For 
a vertex $v$ of $\tau$, let us denote by 
$\boldsymbol\uU_v$ the $|v|$th Minkowski power 
of the simplex $\bDelta_{\up(v)}$, where $\up(v)$ 
is the principal order filter of $\tau$ generated 
by $v$ (thus, $\up(v)$ is a union of vertices). 
Given $r, s \in \NN$, we consider the Minkowski 
sum
\begin{equation} \label{eq:def-Qrs}
\qQ_\tau(r,s) = r \sum_{v} \boldsymbol\uU_v + s 
               \bDelta_E
\end{equation}
in $\RR^E$, where the sum ranges over all vertices 
$v$ of $\tau$.

\begin{proposition} \label{prop:y-gen}
The polytope $\qQ_\tau(r,s)$ is unimodularly 
equivalent to the $Y$-type generalized permutohedron 
\[ r \sum_{v} \uU_v + s \Delta_{E \sqcup \{0\}} \]
in $\RR^{E \sqcup \{0\}}$, where the sum ranges 
over all vertices $v$ of $\tau$ and $\uU_v$ is the 
$|v|$th Minkowski power of the simplex 
$\Delta_{\up(v) \sqcup \{0\}}$ in 
$\RR^{E \sqcup \{0\}}$.
\end{proposition}

\begin{proof}
The projection $\rho: \RR^{E \sqcup \{0\}} \to 
\RR^E$ which forgets the zeroth coordinate induces 
a bijection from the hyperplane consisting of all 
points in $\RR^{E \sqcup \{0\}}$ with sum of 
coordinates equal to 1 onto $\RR^E$ and a 
unimodular equivalence between the lattice simplices
$\Delta_{I \sqcup \{0\}} \subset \RR^{E \sqcup \{0\}}$ 
and $\bDelta_I \subset \RR^E$ for every $I \subseteq 
E$. Thus, $\rho$ induces the claimed unimodular 
equivalence as well.
\end{proof}

Let us define a collection $\mathbf{y} = \mathbf{y}
(r,s)$ of nonnegative integers, indexed by nonempty 
subsets of $E \sqcup \{0\}$, by $y_I = 0$ for every 
nonempty $I \subseteq E$ and
\begin{equation} \label{eq:def-y(r,s)}
y_{I \sqcup \{0\}} = \begin{cases}
                      r|v| + s \delta_{I, E}, 
\text{if} \ I = \up(v) \ \text{for some vertex} \ 
        v \ \text{of} \ \tau, \\
s \delta_{I, E}, \text{otherwise}
\end{cases}
\end{equation}
for every $I \subseteq E$, where 
\[ \delta_{I, E} = \begin{cases}
   1, \text{if} \ I = E, \\
   0, \text{ otherwise.} 
	 \end{cases} \]
We have just shown that $\qQ_\tau(r,s)$ can be 
identified with the $Y$-type generalized 
permutohedron $\pP^y(\mathbf{y})$. The following 
statement describes explicitly $\qQ_\tau(r,s)$ 
in terms of linear inequalities.

\begin{proposition} \label{prop:Q(r,s)-ineq}
The polytope $\qQ_\tau(r,s)$ is defined by the 
inequalities $x_e \ge 0$ 
for every $e \in E$ and 
\begin{equation} \label{eq:Q(r,s)-ineq-b}
\sum_{e \in \iI} x_e \le r |\iI| + s 
\end{equation}
for every order ideal $\iI$ of $\tau$. 
In particular, $\qQ_\tau(r,s)$ coincides with the 
preorder polytope $\qQ_\tau$ for $r=1$ and $s=0$. 
\end{proposition}

\begin{proof}
Let $\mathbf{z} = \mathbf{z}(r,s)$ be the 
collection of nonnegative integers, indexed 
by nonempty subsets of $E \sqcup \{0\}$, defined
by~(\ref{eq:zy}). Then, $z_I = 0$ for every 
nonempty $I \subseteq E$ and 
\[ z_{I \sqcup \{0\}} = r \left( \,
   \sum_{v \, : \, \up(v) \subseteq I} |v| 
	 \right) + s \delta_{I,E} \]
for every $I \subseteq E$. In particular,
\begin{equation} \label{eq:zn}
z_{E \sqcup \{0\}} = r \left( \, \sum_v |v| 
\right) + s = rn+s,
\end{equation}
where the sum ranges over all vertices $v$ of 
$\tau$. By~(\ref{eq:Pz}), the unique linear 
equation defining the generalized permutohedron 
$\pP^y(\mathbf{y})$ becomes 
\begin{equation} \label{eq:eliminate}
x_0 + \sum_{e \in E} x_e = rn + s,
\end{equation}
which we will use to eliminate $x_0$ in the 
linear inequalities associated with nonempty 
subsets of $E \sqcup \{0\}$. Clearly, the 
linear inequalities of~(\ref{eq:Pz}) for nonempty
$I \subseteq E$ amount to $x_e \ge 0$ for every 
$e \in E$, since $z_I = 0$ for all such $I$.

We now consider a subset $J \sqcup \{0\}$, where 
$J \subseteq E$; the associated inequality is 
\begin{equation} \label{eq:associate}
x_0 + \sum_{e \in J} x_e \ge 
z_{J \sqcup \{0\}}. 
\end{equation}
Let $\jJ$ be the maximal order filter of $\tau$ 
contained in $J$; in other words, $\jJ$ is the 
union of all vertices $v$ of $\tau$ such that 
$\up(v) \subseteq J$. Then, 
\[ z_{J \sqcup \{0\}} = r \left( \,
   \sum_{v \, : \, \up(v) \subseteq J} |v| 
	 \right) + s \delta_{J,E} = r |\jJ| + s 
	 \delta_{J,E} \]
and hence, in view of 
Equation~(\ref{eq:eliminate}), the inequality 
(\ref{eq:associate}) is equivalent to
\[ \sum_{e \in I} x_e \le r |\iI| + 
   s(1-\delta_{J,E}), \]
where $I$ and $\iI$ are the complements of $J$ 
and $\jJ$ in $E$, respectively. The special case 
when $J=E$ gives a trivial inequality and can be 
ignored. Otherwise, this inequality 
yields~(\ref{eq:Q(r,s)-ineq-b}) when $J = \jJ$
is an order filter of $\tau$ itself, so that $I 
= \iI$ is an order ideal. It is also implied 
by~(\ref{eq:Q(r,s)-ineq-b}) since $I \subseteq 
\iI$ and hence 
\[ \sum_{e \in I} x_e \le \sum_{e \in \iI} x_e. 
   \]
Since complementation induces a one-to-one 
correspendence between order ideals and order 
filters of $\tau$, the proof follows from our 
considerations and Proposition~\ref{prop:y-gen}.
\end{proof}

\begin{proof}[Proof of Theorem~\ref{thm:main-a}]
Combining Propositions~\ref{prop:y-gen} 
and~\ref{prop:Q(r,s)-ineq} shows that $\qQ_\tau
= \qQ_\tau(1,0)$ is a lattice polytope which is 
unimodularly equivalent to a $Y$-type generalized 
permutohedron. The dimension of $\qQ_\tau$ is 
equal to $n$ since it contains the $n$-dimensional 
cube $[0,1]^E$. 
\end{proof}

\subsection{Lattice points}
\label{sec:lattice}

As mentioned in the introduction, we denote by 
$P_\tau$ the set of lattice points of $\qQ_\tau$
and consider it a subposet of the product poset 
$\NN^E$. Thus, for elements $\mathbf{a} = (a_e)$
and $\mathbf{b} = (b_e)$ of $P_\tau$ we have 
$\mathbf{a} \preceq \mathbf{b}$ if and only if 
$a_e \le b_e$ for every $e \in E$. In our 
running example from Figure~\ref{fig1}, $P_\tau$
has cardinality 92.

The poset $P_\tau$ is an order ideal of $\NN^E$. 
In particular, the zero vector $\mathbf{0}$ is 
its minimum element and every interval 
$[\mathbf{0},x]$ is a product of chains. 
Moreover, $P_\tau$ is graded with rank function 
given by the sum of coordinates. The following 
multiplicative property is a direct consequence 
of the relevant definitions.

\begin{proposition} \label{prop:prod}
Let $\tau$ be the disjoint union of two preorders 
$\tau_1$ and $\tau_2$. Then: 
\begin{itemize}
\itemsep=0pt
\item[{\rm (a)}]
$\qQ_\tau = \qQ_{\tau_1} \times \qQ_{\tau_2}$.

\item[{\rm (b)}]
$P_\tau = P_{\tau_1} \times P_{\tau_2}$.
\end{itemize}
\end{proposition}

The polytope $\qQ_\tau$ has no interior lattice 
points, but it does have some once scaled by 2, 
such as the point $(1, 1,\dots,1)$. Let us call 
\emph{upper boundary} of $\qQ_\tau$ the set of 
its points which satisfy at least one of the 
defining inequalities~\eqref{eq:Q-tau-ineq-b} 
as an equality and denote this set by 
$\hat{\partial}(\qQ_\tau)$.

\begin{proposition} \label{prop:upper}
The number of interior lattice points of $2 
\qQ_{\tau}$ is equal to the number of lattice 
points of $\qQ_\tau$ not in $\hat{\partial}
(\qQ_\tau)$.
\end{proposition}

\begin{proof}
There is a bijection between the two sets of 
lattice points, given by translation by the 
vector $(1, 1,\dots,1)$.
\end{proof}

\section{Ehrhart theory}
\label{sec:Ehrhart}

This section studies the Ehrhart polynomial of 
the preorder polytope $\qQ_\tau$ and proves 
Theorem~\ref{thm:main-b}. An explicit 
combinatorial interpretation of the 
$h^\ast$-polynomial of $\qQ_\tau$ is conjectured,
partial results are given in this direction and 
a conjectural generalization of part (b) of 
Theorem~\ref{thm:main-b} (Ehrhart-Zeta duality) 
is included.

\subsection{Ehrhart-Zeta duality}
\label{sec:EZ}

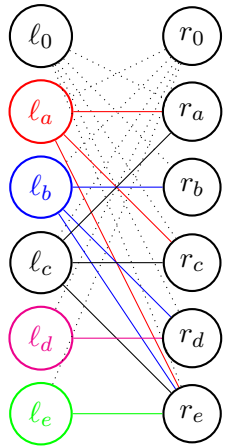
\begin{figure}[!h]
\centering
\begin{tikzpicture}[scale=0.5]
\tikzstyle{every node}=[draw,shape=circle, thick]

\node [draw, color=green] at (0,0) (e) {$\ell_e$};
\node [draw, color=magenta] at (0,2) (d) {$\ell_d$};
\node at (0,4) (c) {$\ell_c$};
\node [draw, color=blue] at (0,6) (b) {$\ell_b$};
\node [draw, color=red] at (0,8) (a) {$\ell_a$};
\node at (0,10) (l0) {$\ell_0$};

\node at (4,0) (re) {$r_e$};
\node at (4,2) (rd) {$r_d$};
\node at (4,4) (rc) {$r_c$};
\node at (4,6) (rb) {$r_b$};
\node at (4,8) (ra) {$r_a$};
\node at (4,10) (r0) {$r_0$};

\path [draw, dotted] (l0) -- (ra);
\path [draw, dotted] (l0) -- (rb);
\path [draw, dotted] (l0) -- (rc);
\path [draw, dotted] (l0) -- (rd);
\path [draw, dotted] (l0) -- (re);

\path [draw, dotted] (r0) -- (a);
\path [draw, dotted] (r0) -- (b);
\path [draw, dotted] (r0) -- (c);
\path [draw, dotted] (r0) -- (d);
\path [draw, dotted] (r0) -- (e);

\path [draw, color=red] (a) -- (ra);
\path [draw, color=red] (a) -- (rc);
\path [draw, color=red] (a) -- (re);

\path [draw, color=blue] (b) -- (rb);
\path [draw, color=blue] (b) -- (rd);
\path [draw, color=blue] (b) -- (re);

\path [draw] (c) -- (ra);
\path [draw] (c) -- (rc);
\path [draw] (c) -- (re);

\path [draw, color=magenta] (d) -- (rd);

\path [draw, color=green] (e) -- (re);

\end{tikzpicture}
\caption{Example of bipartite graph $G_\tau$.}
\label{fig2}
\end{figure}

Let us define a bipartite graph $G_\tau$ with 
$n+1$ left and $n+1$ right vertices as follows. 
There are vertices $\ell_e$ on the left side 
and $r_e$ on the right side, for $e \in E 
\sqcup \{0\}$. Vertex $\ell_0$ is connected by 
an edge to all vertices on the right and $r_0$ 
is connected by an edge to all vertices on the 
left. For every $e \in E$, there are edges from 
$\ell_e$ to all vertices $r_f$ such that $f$ 
belongs to the principal order filter $\up(e)$ 
of $\tau$. The graph $G_\tau$ has no isolated 
vertex; it is shown in Figure~\ref{fig2} in the 
case of our running example from 
Figure~\ref{fig1}. We set
\[ I_e = \{f \in E \sqcup \{0\} \mid \{\ell_e, 
   r_f\} \ \text{is an edge of} \ G_\tau \} \]
for $e \in E \sqcup \{0\}$. Equivalently, $I_e 
= \up(e) \sqcup \{0\}$ for every $e \in E$ and 
$I_0 = E \sqcup \{0\}$. According to 
Propositions~\ref{prop:y-gen} 
and~\ref{prop:Q(r,s)-ineq}, the preorder polytope
$\qQ_\tau$ is unimodularly equivalent to the 
$Y$-type generalized permutohedron
\[ \sum_{e \in E \sqcup \{0\}} y_e \Delta_{I_e} 
   \]
in $\RR_{\ge 0}^{E \sqcup \{0\}}$ with parameters 
$y_0 = 0$ and $y_e = 1$ for every $e \in E$. To 
apply Theorem~\ref{thm:post}, we need to 
understand the $G_\tau$-draconian sequences. 
\begin{lemma} \label{lem:tau-drac}
A sequence $(a_e)_{e \in E \sqcup \{0\}} \in 
\NN^{E \sqcup \{0\}}$ is $G_\tau$-draconian if 
and only if 
\begin{equation} \label{eq:draco1}
a_0 + \sum_{e \in E} a_e = n
\end{equation}
and $(a_e)_{e \in E}$ is a lattice point of 
$\qQ_{\tau^\ast}$.
\end{lemma}

\begin{proof}
According to the definition (see 
Section~\ref{sec:genperm}), $(a_e)_{e \in E 
\sqcup \{0\}} \in \NN^{E \sqcup \{0\}}$ is 
$G_\tau$-draconian if and only if 
Equation~\eqref{eq:draco1} is valid 
and~\eqref{eq:draco2} holds for every nonempty 
$\aA \subseteq E \sqcup \{0\}$. Let us analyze 
these inequalities and show that some of them 
are superflous.

Since $I_0 = E \sqcup \{0\}$, if $0 \in \aA$ 
then \eqref{eq:draco2} becomes $\sum_{e \in \aA} 
a_e \le n$, which is a consequence 
of~\eqref{eq:draco1}. Thus, we may assume that 
$\aA \subseteq E$. Then, we get
\begin{equation} \label{eq:minidraco1}
  \sum_{e \in \aA} a_e \leq |\up(\aA)|,
\end{equation}
where $\up(\aA)$ is the order filter of $\tau$ 
generated by $\aA$. In particular, if $\aA$ is 
an order filter $\jJ$ itself, so that $\up(\aA) 
= \jJ$, we get
\begin{equation} \label{eq:minidraco2}
\sum_{e \in \jJ} a_e \le |\jJ|.
\end{equation}
In general,
$ \sum_{e \in \aA} a_e \le \sum_{e \in \up(\aA)} 
   a_e$,
so \eqref{eq:minidraco1} follows from 
\eqref{eq:minidraco2} for $\jJ = \up(\aA)$. As
a conclusion, the $G_\tau$-draconian conditions 
are equivalent to Equation~\eqref{eq:draco1} 
and the validity of~\eqref{eq:minidraco2} for
all order filters $\jJ$ of $\tau$. The latter 
means that $(a_e)_{e \in E}$ is a lattice point 
of $\qQ_{\tau^\ast}$ and the proof follows. 
\end{proof}

We are now in a position to prove 
Theorem~\ref{thm:main-b}.
\begin{proof}[Proof of Theorem~\ref{thm:main-b}]
The formula of part (a) for the Ehrhart 
polynomial follows directly by combining 
Theorem~\ref{thm:post}, applied to the graph 
$G_\tau$ for the parameters $y_0 = 0$ and $y_e = 
1$ for every $e \in E$, with 
Lemma~\ref{lem:tau-drac}. 

To deduce part (b), let us count the 
$m$-multichains in the poset $P_{\tau^\ast}$ for 
$m \ge 1$. One can can condition on the top element 
$\mathbf{a} = (a_e)_{e \in E} \in P_{\tau^\ast}$ of 
such a multichain. Since the interval 
$[0,\mathbf{a}]$ is isomorphic to a product of 
chains of lengths $a_e$, the number of 
$m$-multichains in $P_{\tau^\ast}$ with top 
element $\mathbf{a}$ is equal to the product 
$\prod_{e \in E} \binom{m+a_e-1}{a_e}$. As a 
result, 
\[ \zZ(P_{\tau^\ast},m+1) = 
   \sum_{(a_e) \in P_{\tau^\ast}} \prod_{e \in E} 
	 \binom{m+a_e-1}{a_e} \]
for every $m \ge 1$ and part (b) follows by 
comparison with the formula of part (a).

We may now deduce the last statement of part (a) 
as follows. By part (b), the normalized volume 
of $\qQ_\tau$ is equal to $n!$ times the 
coefficient of $t^n$ in $\zZ(P_{\tau^\ast},t+1)$ 
and hence to the number of chains of 
$P_{\tau^\ast}$ of length $n$ (since 
$P_{\tau^\ast}$ is graded of rank $n$, these are 
exactly the maximal chains of $P_{\tau^\ast}$).
To every such chain, one can associate the word 
of length $n$ in the alphabet $E$ which records 
at each step of the chain the unique coordinate 
which increases. The resulting map is one-to-one
and, by definition of $\qQ_{\tau^\ast}$, the 
words thus obtained are those words $w \in E^n$ 
such that the elements of any order ideal $\jJ$ 
of $\tau^\ast$ appear at most $|\jJ|$ times in 
$w$. Equivalently, the elements of any order 
ideal $\iI$ of $\tau$ should appear at least 
$|\iI|$ times in $w$. 
\end{proof}

\subsection{Matrix Ehrhart-Zeta duality}
\label{sec:matrixEZ}

One can extend Ehrhart-Zeta duality in terms of 
infinite matrices as follows. For $k, \ell \in 
\NN$ we denote by $\nabla_\tau(k,\ell)$ the 
number of $k$-multichains in the poset of 
lattice points of $\ell \qQ_\tau$ (where the 
order is induced from $\NN^n$, as usual). This
function 
can be viewed as an infinite matrix with rows
and columns indexed by nonnegative integers, 
every column of which is the sequence of values 
of a zeta polynomial. We have $\nabla_\tau(k,0) 
= \nabla_\tau(0,\ell) = 1$, $\nabla_\tau(k,1) = 
\zZ(P_\tau,k+1)$ and $\nabla_\tau(1,\ell) = 
\Ehr(\qQ_\tau,\ell)$ for all $k, \ell \in \NN$.

The following conjecture generalizes part (b) 
of Theorem~\ref{thm:main-b}; it has been verified 
by computer on certain finite pieces of 
$\nabla_\tau$ for preorders of size at most 5.
\begin{conjecture} \label{conj:matrixEZ}
The matrix $\nabla_\tau$ is the transpose of
$\nabla_{\tau^\ast}$ for every preorder $\tau$.
\end{conjecture}

A bijective proof of Ehrhart-Zeta duality 
was given for chain preorders in~\cite{Ath26}.
This proof can be easily extended in the setting 
of the matrix Ehrhart-Zeta duality. We sketch the 
proof of the following proposition, which 
generalizes the first proof of 
Theorem~\ref{thm:Ath26} given in 
\cite[Section~4]{Ath26}.
\begin{proposition} \label{prop:matrixEZ}
Conjecture~\ref{conj:matrixEZ} holds for all
chain preorders.
\end{proposition}

\begin{proof}
Let $\tau$ be a chain $v_1 \prec v_2 \prec 
\cdots \prec v_p$ with vertices of sizes $r_1, 
r_2,\dots,r_p$, respectively. We set $s_i = r_1 
+ r_2 + \cdots + r_i$ for $0 \le i \le p$, so
that where $s_0 := 0$ and $s_p = n$. We may 
assume that $E = \{1, 2,\dots,n\}$ with elements
indexed so that $\qQ_\tau$ is described by the 
inequalities $x_i \ge 0$ for $1 \le i \le n$ and 
\[ x_1 + x_2 + \cdots + x_{s_i} \le s_i \]
for $1 \le i \le p$. 
By \cite[Lemma~4.1]{Ath26}, the $k$-multichains 
in the poset of lattice points of $\ell \qQ_\tau$ 
are in bijection with the points 
$(a_1, a_2,\dots,a_{nk}) \in \NN^{nk}$ such that 
\begin{equation} \label{eq:ineq-ai}
a_1 + a_2 + \cdots + a_{s_i k} \le s_i \ell
\end{equation}
for $1 \le i \le p$. Given such a point, for 
$0 \le j \le n\ell$ we define $b_j$ as the 
number of $i \in \{1, 2,\dots,nk\}$ such 
that $a_1 + a_2 + \cdots + a_i = j$. Then, 
\begin{equation} \label{eq:sum-bj}
b_0 + b_1 + \cdots + b_{n \ell} = nk
\end{equation}
and 
\begin{equation} \label{eq:ineq-bj-1}
b_0 + b_1 + \cdots + b_{s_i \ell} \ge s_i k 
\end{equation}
for $0 \le i \le p$. Using 
Equation~(\ref{eq:sum-bj}) to eliminate $b_0$ 
we find that the solutions $(a_1, a_2,\dots,a_{nk}) 
\in \NN^{nk}$ of the system of linear 
inequalities (\ref{eq:ineq-ai}) are in bijection 
with the solutions $(b_1, b_2,\dots,b_{n \ell}) 
\in \NN^{n\ell}$ of the system of linear 
inequalities 
\begin{equation} \label{eq:ineq-bj-2}
b_{s_i \ell + 1} + b_{s_i \ell + 2} + \cdots + 
b_{n \ell} \le (n - s_i) k
\end{equation}
for $0 \le i \le p-1$. By \cite[Lemma~4.1]{Ath26},
and since $\qQ_{\tau^\ast}$ is described by the 
inequalities $x_i \ge 0$ for $1 \le i \le n$ and 
\[ x_{s_i + 1} + x_{s_i + 2} + \cdots + x_n 
       \le n - s_i \]
for $0 \le i \le p-1$, these solutions are in 
bijection with the $\ell$-multichains in the 
poset of lattice points of $k \qQ_{\tau^\ast}$
and the proof follows.
\end{proof}

\subsection{On the $h^\ast$-polynomial}
\label{sec:h-ast}

We recall from Section~\ref{sec:Ehr-hstar} that 
the $h^\ast$-polynomial $h^\ast(\qQ_\tau, t)$
has nonnegative coefficients which sum to the 
normalized volume of $\qQ_\tau$. Thus, the 
following conjecture is a natural refinement of
the second statement of Theorem~\ref{thm:main-b}
(a). We fix an arbitrary total order $\le$ on the ground 
set $E$ and, given a word $w = (w_1, 
w_2,\dots,w_n) \in E^n$, we denote by $\des(w)$ 
the number of \emph{descents} of $w$, meaning 
indices $i \in \{1, 2,\dots,n-1\}$ such that 
$w_i > w_{i+1}$. The following conjecture extends
\cite[Conjecture~5.2]{AXY26} from arbor polytopes
to all preorder polytopes.
\begin{conjecture} \label{conj:h-ast}
Let $\tau$ be a preorder on a totally ordered 
$n$-element set $E$. Then
\[ h^\ast(\qQ_\tau, t) = \sum_{w \in \wW_\tau}
       t^{n-1-\des(w)}, \]
where $\wW_\tau$ is the set of words $w \in E^n$ 
such that for every order ideal $\iI$ of $\tau$, 
the elements of $\iI$ appear a total of at least 
$|\iI|$ times in $w$. 
\end{conjecture}

This conjecture has been verified by computer
for all preorders of size at most 7. The motivation 
behind it comes from \cite[Theorem~1.5]{AXY26}, 
which verifies it in the special case in which 
$\tau$ is the preorder associated to an arbor 
which consists of a root of any size and several 
single element leaves.

A combinatorial interpretation of 
$h^\ast(\qQ_\tau, t)$ which seems different from
that of Conjecture~\ref{conj:h-ast} was given in
\cite[Section~4]{Ath26} in the special case that
the preorder $\tau$ is a chain. The proof is 
based on the theory of pure shellability of posets 
and generalizes to all preorders with a minimum 
vertex. To state the result, it is convenient to 
assume that $E = \{1, 2,\dots,n\}$ is totally 
ordered as usual. Given a word $w = 
(w_1, w_2,\dots,w_n) \in \{1, 2,\dots,n\}^n$, we 
denote by $\asc^\ast(w)$ the number of indices 
(called \emph{modified ascents}) $i \in \{0, 
1,\dots,n-1\}$ such that $w_i < w_{i+1}$, where 
$w_0 := 1$. The following statement generalizes 
\cite[Proposition~4.2]{Ath26}. To avoid making 
too long of a digression we assume familiarity 
with the relevant parts of 
\cite[Section~4]{Ath26} and only sketch the 
proof, which directly extends that of 
\cite[Proposition~4.2]{Ath26}.
\begin{proposition} \label{prop:h-ast-min} 
Let $\tau$ be a preorder of size $n$ on the 
ground set $E = \{1, 2,\dots,n\}$ which has a 
minimum vertex. Assuming that 1 belongs to this
vertex we have
\[ h^\ast(\qQ_\tau, t) = \sum_{w \in \wW_\tau} 
   t^{\asc^\ast(w)}, \]
where $\wW_\tau$ is as in 
Conjecture~\ref{conj:h-ast}.
\end{proposition}

\begin{proof}
As explained in the proof of 
\cite[Proposition~4.2]{Ath26} we have 
\[ h^\ast(\qQ_\tau, t) = h(\Delta(P_{\tau^\ast}), 
   t) \]
by Ehrhart-Zeta duality, where $\Delta
(P_{\tau^\ast})$ stands for the order complex of 
$P_{\tau^\ast}$. 

We consider the poset 
$\widehat{P}_{\tau^\ast} = P_{\tau^\ast} \cup 
\{ \hat{1} \}$, obtained from $P_{\tau^\ast}$
by adding a maximum element $\hat{1}$, and an 
integer valued edge labeling $\lambda$ of this 
poset, defined as follows. For a cover 
relation $(u,v)$ of $P_{\tau^\ast}$ we set 
$\lambda(u,v) = i$, if $u$ and $v$ differ in 
the $i$th coordinate. We also set 
$\lambda(u, \hat{1}) = 1$ for every maximal
element $u \in P_\tau$ and claim that $\lambda$ 
is an EL-labeling. This means that, given any 
nonsingleton closed interval $[u, v]$ of 
$\widehat{P}_{\tau^\ast}$, there is a unique 
rising (with respect to $\lambda$) maximal 
chain in $[u, v]$ and that this chain is 
lexicographically least among all maximal 
chains of $[u, v]$. This is clear if $v \ne 
\hat{1}$, since then the labels of the maximal 
chains of $[u, v]$ are the permutations of a 
multiset. Suppose that $v = \hat{1}$. By 
assumption, $\tau^\ast$ has a maximum vertex 
and this vertex contains 1. This implies that
from the linear inequalities 
(\ref{eq:Q-tau-ineq-b}) which define 
$\qQ_{\tau^\ast}$, only the one associated 
to the order ideal $E$ of $\tau^\ast$ involves 
the variable $x_1$. As a result, there is a 
unique maximal chain in $[u, \hat{1}]$ with 
label $(1, 1,\dots,1)$ and this chain is the 
only rising maximal chain of $[u, \hat{1}]$. 

Since $\lambda$ is EL, 
$h(\Delta(P_{\tau^\ast}), t)$ can be 
interpreted as the descent enumerator (with 
respect to $\lambda$) of the maximal chains 
of $\widehat{P}_{\tau^\ast}$. As already 
observed in the proof of 
Theorem~\ref{thm:main-b}, these chains are in 
bijection with the elements of $\wW_\tau$ and 
the proof follows.
\end{proof}

The proof of Proposition~\ref{prop:h-ast-min}
shows that the poset $P_\tau$ is shellable, and 
hence Cohen-Macaulay (over $\ZZ$ and any field), 
if $\tau$ has a maximum vertex. This raises the
following question.

\begin{question} \label{que:shellability}
Is the poset $P_\tau$ shellable for every 
preorder $\tau$? Is it Cohen--Macaulay for 
every preorder $\tau$?
\end{question}

A formula for $h^\ast(\qQ_\tau, t)$ may be 
derived from part (a) of Theorem~\ref{thm:main-b}
as follows.
\begin{proposition} \label{prop:h-ast}
Let $\tau$ be a preorder of size $n$ on a 
totally ordered ground set $E$. Then
\[ h^\ast(\qQ_\tau, t) = (1-t)^n + 
   \sum_{w \in \widehat{\wW}_\tau \smallsetminus 
	 \{\varnothing\}} (1-t)^{n-\ell(w)} 
	 t^{\des(w)}, \]
where $\widehat{\wW}_\tau$ is the set of words 
$w$ of length (number of letters) $\ell(w) \le n$ 
with letters from $E$ such that for every order 
filter $\jJ$ of $\tau$, the elements of $\jJ$ 
appear a total of at most $|\jJ|$ times in $w$. 
\end{proposition}

\begin{proof}
Given $\mathbf{a} = (a_e) \in \NN^E$, let us 
denote by $\mathfrak{S}_{\mathbf{a}}$ the set 
of permutations of the multiset which consists
of $a_e$ copies of $e$, for $e \in E$. By part 
(a) of Theorem~\ref{thm:main-b},
\[ \begin{aligned}
   \sum_{m \ge 0} \Ehr(\qQ_\tau,m) t^m &= 1 + 
	 \sum_{m \ge 0} \Ehr(\qQ_\tau,m+1) t^{m+1} \\
&= 1 + t \sum_{m \ge 0} \left( \, 
   \sum_{(a_e) \in P_{\tau^\ast}} \prod_{e \in E} 
	 \binom{m+a_e}{a_e} \right) t^m \\ 
&= 1 + t \sum_{\mathbf{a} \in P_{\tau^\ast}}
   \left( \, \sum_{w \in \mathfrak{S}_{\mathbf{a}}} 
	 t^{\des(w)} \right) / (1-t)^{1 + |\mathbf{a}|}, 
\end{aligned} \]
where the third equality applies MacMahon's 
formula \cite[p.~211]{Mac04}
\[ \sum_{m \ge 0} \left( \, \prod_{e \in E} 
   \binom{m+a_e}{a_e} \right) t^m
   = \frac{\displaystyle\sum_{w \in 
	   \mathfrak{S}_{\mathbf{a}}} 
	 t^{\des(w)}}{(1-t)^{1 + |\mathbf{a}|}} \]
for $\mathbf{a} = (a_e) \in \NN^E$. This 
computation shows that 
\[ h^\ast(\qQ_\tau, t) = (1-t)^{n+1} + t
   \sum_{\mathbf{a} \in P_{\tau^\ast}} 
	 (1-t)^{n-|\mathbf{a}|} 
	 \sum_{w \in \mathfrak{S}_{\mathbf{a}}} 
	 t^{\des(w)}. \]
We now observe that the summand of the outer 
sum on the right-hand side for $\mathbf{a} = 
\mathbf{0}$ is equal to $(1-t)^n$ and that
$\widehat{\wW}_\tau$ is the disjoint union of 
the sets $\mathfrak{S}_{\mathbf{a}}$ for 
$\mathbf{a} \in P_{\tau^\ast}$ and the proof 
follows.
\end{proof}

As mentioned in the introduction, Ehrhart
polynomials of preorder polytopes are not always
real-rooted. For example, if $\tau$ is a partial 
order consisting of a minimum element and two 
other elements incomparable to each other which 
cover it, then $\Ehr(\qQ_\tau, t) = 1 + 23t/6 + 
5t^2 + 13t^3/6$ has two nonreal complex roots.
We conjecture that $\Ehr(\qQ_\tau, t)$ has the
weaker magic positivity property and hence that
$h^\ast(\qQ_\tau,t)$ is real-rooted.
\begin{conjecture} \label{conj:magic}
The Ehrhart polynomial $\Ehr(\qQ_\tau, t)$ is 
magic positive for every preorder $\tau$. In 
particular, $h^\ast(\qQ_\tau,t)$ is real-rooted. 
\end{conjecture}

This conjecture has been proven for the special 
class of the arbor polytopes studied in~\cite{AXY26} 
and also follows from \cite[Theorem~13]{SP02} 
when $\tau$ is a chain with vertices of size one 
(in which case $\qQ_\tau$ is the standard 
Pitman--Stanley polytope). The coefficients of 
the Ehrhart polynomial in the magic basis 
$t^i(1+t)^{n-i}$, multiplied by $n!$, have 
interesting combinatorial interpretations in both 
cases; see \cite[Section~1]{AFM26}
\cite[Section~3]{AXY26} \cite[Theorem~10.1]{GS06}. 
Conjecture~\ref{conj:magic} has been verified by 
computer for all preorders of size at most 7.

\subsection{On the zeta polynomial}
\label{sec:zeta-conj}

Computational evidence suggests that $(-1)^n 
\zZ(P_\tau,-1)$ is equal to the number of maximal
elements of $P_\tau$. A refined conjecture may be 
formulated using the notion of $q$-zeta polynomial
introduced in~\cite{Cha25a}. Given a finite graded
poset $P$ with rank function $\rho: P \to \NN$, it 
was shown in~\cite{Cha25a} that there exists a 
polynomial $\zZ_q(P,t) \in \QQ(q)[t]$ such that 
\[ \zZ_q(P,[m]_q) = \sum_{p_1 \preceq p_2 \preceq 
   \cdots \preceq p_{m-1}} q^{\rho(p_1) + \rho(p_2) 
	 + \cdots + \rho(p_{m-1})} \]
for every integer $m \ge 2$, where the sum ranges 
over all multichains $p_1 \preceq p_2 \preceq 
\cdots \preceq p_{m-1}$ in $P$ and $[m]_q = 
(q^m-1)/(q-1)$ is the standard $q$-analog of $m$. 
Let us denote by $\Max(P)$ the set of maximal 
elements of a finite poset $P$.
\begin{conjecture} \label{conj:zeta}
For every preorder $\tau$ of size $n$ 
\begin{equation} \label{eq:conj-qzeta}
\zZ_q(P_\tau, [-1]_q) = 
(-1)^n \sum_{\mathbf{a} \in \Max(P_\tau)} 
q^{-\cov(\mathbf{a})},
\end{equation}
where $\cov(\mathbf{a})$ is the number of elements 
of $P_\tau$ covered by $\mathbf{a}$ and $[-1]_q = 
-1/q$. In particular,
\begin{equation} \label{eq:conj-zeta}
\zZ(P_\tau,-1) = (-1)^n \, \# \Max(P_\tau) 
\end{equation}
for every preorder $\tau$ of size $n$.  
\end{conjecture}

This conjecture has been verified by
computer for all preorders of size at most 6. 
For the example of Figure~\ref{fig1} we have
$\zZ(P_\tau,-1) = -18$ and
\[ \zZ_q(P_\tau, [-1]_q) = - (2q + 1)
   (q^2 + 4q + 1) / q^5 \]
and $P_\tau$ has 18 maximal elements.

\begin{remark} \label{rem:conj-zeta} \rm
(a) Equation~(\ref{eq:conj-zeta}) is a consequence
of Conjecture~\ref{conj:h-ast}. Indeed, we have 
$\zZ(P_\tau,-1) = \Ehr(\qQ_{\tau^\ast},-2)$ by 
Ehrhart-Zeta duality and 
\[ \Ehr(\qQ_{\tau^\ast}, t) = \sum_{i=0}^n 
   h^\ast_i(\qQ_{\tau^\ast}) \binom{t+n-i}{n} \]
by Equation~(\ref{eq:def-hstar}). Moreover, 
$h^\ast_n(\qQ_{\tau^\ast}) = 0$ since 
$\qQ_{\tau^\ast}$ has no interior lattice points
and hence 
\[ \Ehr(\qQ_{\tau^\ast},-2) = (-1)^n \,
   h^\ast_{n-1}(\qQ_{\tau^\ast}). \]
According to Conjecture~\ref{conj:h-ast}, 
$h^\ast_{n-1}(\qQ_{\tau^\ast})$ is equal to the 
number of words $w \in \wW_{\tau^\ast}$ with no
descent. These are the (weakly) increasing words
$w \in E^n$ such that the elements of any order 
filter $\jJ$ of $\tau$ appear a total of at least 
$|\jJ|$ times in $w$. Equivalently, the elements 
of any order ideal $\iI$ of $P_\tau$ should appear 
at most $|\iI|$ times in $w$. These words are 
obviously in bijection with the maximal elements 
of $P_\tau$.  

(b) By Ehrhart reciprocity, $(-1)^n \,
\Ehr(\qQ_{\tau^\ast},-2)$ is equal to the number 
of lattice points in the interior of 
$2\qQ_{\tau^\ast}$. Thus, in view of 
Proposition~\ref{prop:upper}, 
Equation~(\ref{eq:conj-zeta}) is equivalent to 
the statement that the number of maximal elements
of $P_\tau$ is equal to the number of lattice 
points of $\qQ_{\tau^\ast}$ not in the upper 
boundary $\hat{\partial} (\qQ_{\tau^\ast})$.
\qed
\end{remark}

\begin{remark} \label{rem:conj-zeta-chains} \rm
Equation~(\ref{eq:conj-zeta}) holds for every
chain preorder $\tau$. To explain why, let us 
adopt the notation of the proof of 
Proposition~\ref{prop:matrixEZ}. By part (b) 
of Remark~\ref{rem:conj-zeta}, it suffices to
show that the number of maximal elements
of $P_{\tau^\ast}$ is equal to the number of 
lattice points of $\qQ_\tau \smallsetminus 
\hat{\partial} (\qQ_\tau)$. The maximal elements
of $P_{\tau^\ast}$ are the tuples $(a_1, 
a_2,\dots,a_n) \in \NN^n$ such that 
$a_1 + a_2 + \cdots + a_n = n$
and 
\[ a_{s_i + 1} + a_{s_i + 2} + \cdots + 
a_n \le n - s_i \]
for $1 \le i \le p-1$ or, equivalently, 
\[ a_1 + a_2 + \cdots + a_{s_i} \ge s_i \]
for $1 \le i \le p-1$. The lattice points 
of $\qQ_\tau \smallsetminus \hat{\partial} 
(\qQ_\tau)$ are the tuples $(b_1, b_2,\dots,b_n) 
\in \NN^n$ such that 
\[ b_1 + b_2 + \cdots + b_{s_i} < s_i \]
for $1 \le i \le p$. Given a maximal element
$(a_1, a_2,\dots,a_n)$ of $P_{\tau^\ast}$, for 
$1 \le j \le n$ we define $b_j$ to be the 
number of indices $i \in \{1, 2,\dots,n\}$ 
such that $a_1 + a_2 + \cdots + a_i = j-1$. 
We leave it to the reader to verify that 
$(b_1, b_2,\dots,b_n)$ is a lattice point of 
$\qQ_\tau \smallsetminus \hat{\partial} 
(\qQ_\tau)$ and that the resulting map is a 
bijection from the set of maximal elements of 
$P_{\tau^\ast}$ to the set of lattice points 
of $\qQ_\tau \smallsetminus \hat{\partial} 
(\qQ_\tau)$. 
\end{remark}


\section{Lattice point enumeration}
\label{sec:h-vectors}

This section extends to preorder polytopes some 
conjectures of~\cite{Cha25b} on the lattice point 
enumeration of arbor polytopes. Throughout this 
and the following two sections, $\tau$ continues 
to be a preorder of size $n$ on the ground set $E$.

For $0 \le i \le n$, we let $h_i(\tau)$ be  
the number of lattice points of $\qQ_\tau$ which 
have exactly $i$ nonzero coordinates. The vector 
$h(\tau) = (h_0(\tau), h_1(\tau),\dots,h_n(\tau))$ 
is called the \emph{$h$-vector} of $\tau$ and the 
polynomial $h(\tau, t) = \sum_{i=0}^n h_i(\tau)t^i$ 
is called the \emph{$h$-polynomial}. These concepts
directly generalize the notion of $h$-vector of an
arbor, introduced in \cite{Cha25b}. For the example 
of Figure~\ref{fig1} we have $h(\tau) = (1, 11, 34, 
34, 11, 1)$. We note that $h_0(\tau) = h_1(\tau) = 
1$. The following conjecture, the parts of 
which appear in the order of increasing strength, 
generalizes Conjecture~1.1 in~\cite{Cha25b}. For 
undefined terminology about simplicial polytopes, 
see \cite{StaCCA}.
\begin{conjecture} \label{conj:hvector}
For every preorder $\tau$ of size $n$:
\begin{itemize}
\itemsep=0pt
\item[{\rm (a)}]
$h(\tau)$ is palindromic, meaning that $h_i(\tau) 
= h_{n-i}(\tau)$ for $0 \le i \le n$.

\item[{\rm (b)}]
$h(\tau)$ is palindromic and unimodal, meaning 
that $h_i(\tau) = h_{n-i}(\tau)$ for $0 \le i \le 
n$ and $h_0(\tau) \le h_1(\tau) \le \cdots \le 
h_{\lfloor n/2 \rfloor}(\tau)$.

\item[{\rm (c)}]
$h(\tau)$ is equal to the $h$-vector of an 
$n$-dimensional simplicial polytope.

\item[{\rm (d)}]
$h(\tau)$ is equal to the $h$-vector of an 
$n$-dimensional flag simplicial polytope.
\end{itemize}

\end{conjecture}

For example, if $\tau$ is the chain preorder
of size $n$ having $n$ vertices of cardinality 
1 or one vertex of cardinality $n$, then the 
role of the flag simplicial polytope of part (d)
can be played by the simplicial associahedron of 
of type $\TA_n$ and $\TB_n$, respectively; see 
Sections~7 and~8 of \cite{Cha25b}. The following 
statement was implicitly conjectured for arbor 
polytopes in \cite[Section~1.1]{Cha25b}.
\begin{conjecture} \label{conj:gamma}
The polynomial $h(\tau,t)$ is $\gamma$-positive,
meaning that 
\[ h(\tau,t) = \sum_{i=0}^{\lfloor n/2 \rfloor}
\gamma_i(\tau) t^i (1+t)^{n-2i} \]
for some nonnegative integers $\gamma_0(\tau), 
\gamma_1(\tau),\dots,\gamma_{\lfloor n/2 \rfloor}
(\tau)$, for every preorder $\tau$ of size $n$.
In particular, the vector $h(\tau)$ is palindromic 
and unimodal.
\end{conjecture}

A more optimistic conjecture is the following.
\begin{conjecture} \label{conj:roots}
The polynomial $h(\tau,t)$ has only real roots 
for every preorder $\tau$.
\end{conjecture}

The following statement substantially 
generalizes Conjecture~4.2 in~\cite{Cha25b}.
\begin{conjecture} \label{conj:hvector-dual}
We have $h(\tau) = h(\tau^\ast)$ for every 
preorder $\tau$.
\end{conjecture}

All previous conjectures, other than 
Conjecture~\ref{conj:roots}, were proven for 
chain preorders in \cite{Ath26}. They have been 
verified (together with 
Conjecture~\ref{conj:roots}) in several other 
special cases in \cite{AXY26} 
\cite[Sections~7-10]{Cha25b}; see also 
Section~\ref{sec:exa}. The authors have verified 
by computer Conjecture~\ref{conj:hvector} (c) for 
all preorders of size at most 7 and arbors of
size at most 8, Conjectures~\ref{conj:gamma} 
and~\ref{conj:hvector-dual} for all preorders 
of size at most 7 and 
Conjecture~\ref{conj:roots} for all preorders 
of size at most 7 and arbors of size at most 10.




\begin{remark} \label{rem:christos-example} \rm
We do not know of a class of generalized 
permutohedra larger than that of preorder 
polytopes for which the $h$-vector may be 
palindromic. For example, $\qQ = 
\bDelta_{\{1,2\}} + \bDelta_{\{2,3\}} + 
\bDelta_{\{3\}}$ is a $Y$-type generalized 
permutohedron  in $\RR^3$ described by the 
inequalities $x_1, x_2, x_3 \ge 0$ and 
\begin{align*} 
x_1 & \le 1 \\  
x_1 + x_2 & \le 2 \\
x_3 & \le 2 \\
x_1 + x_2 + x_3 & \le 3.
\end{align*}
Counting its lattice points by the number of 
nonzero coordinates gives the vector $(1,5,6,1)$,
which is not palindromic.
\end{remark}

\begin{remark} \label{rem:series} \rm
One may naturally consider the formal power 
series
\[ F(\tau, t) = 
   \frac{h(\tau,t)}{{(1-t)}^{n+1}}. \]
The coefficients of $F(\tau, t)$ are the values 
of a polynomial at nonnegative integers. Curiously,
this polynomial has all its complex roots on the 
line $\re(z) = -1/2$ for all preorders of size at 
most 7 and all arbors of size at most 8. The arbor 

\medskip
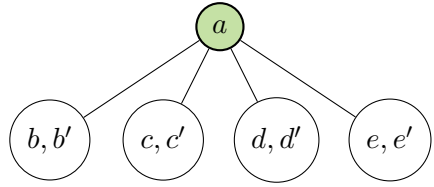
\begin{center}
\begin{tikzpicture}
[grow=down]
\tikzstyle{every node}=[draw,shape=circle]

\node [fill=racines,thick] {$a$}
  child {node {$b,b'$}}
  child {node {$c,c'$}}
  child {node {$d,d'$}}
  child {node {$e,e'$}}
;
\end{tikzpicture}
\end{center}

\medskip
\noindent
is the unique arbor of size 9 for which this
property fails.
%
\end{remark}

\section{Preorder polytopes and reflexive polytopes}
\label{sec:reflexive}

This section shows that, after suitable 
translation, the lattice polytope 
\[ \qQ_\tau(1,1) = \qQ_\tau + \bDelta_E \] 
(see Section~\ref{sec:Minkowski}) is reflexive 
and investigates its properties. We recall that 
an $n$-dimensional lattice polytope $\qQ$ in 
$\RR^n$ having $\mathbf{0}$ in its  
interior is said to be \emph{reflexive} if the 
polar dual of $\qQ$ is also a lattice polytope. 
An equivalent condition \cite[Theorem~35.8]{HiAC} 
is that the $h^\ast$-polynomial of $\qQ$ is 
palindromic (of degree $n$).

Let $\rR_\tau$ denote the translation of
$\qQ_\tau(1,1)$ by the negative of the vector 
$\mathbf{1} = \sum_{e \in E} {\rm \mathbf{e}}_e
\in \RR^E$. By Proposition~\ref{prop:Q(r,s)-ineq}, 
$\qQ_\tau(1,1)$ is defined by the inequalities 
$x_e \ge 0$ for every $e \in E$ and 
\[ \sum_{e \in \iI} x_e \le |\iI| + 1 \] 
for every order ideal $\iI$ of $\tau$. As a result, 
$\rR_\tau$ is defined by the inequalities 
\begin{equation} \label{eq:Rtau-a}
x_e \ge -1 
\end{equation}
for every $e \in E$ and 
\begin{equation} \label{eq:Rtau-b}
\sum_{e \in \iI} x_e \le 1
\end{equation}
for every order ideal $\iI$ of $\tau$. 
\begin{proposition}
The lattice polytope $\rR_\tau$ is reflexive for 
every preorder $\tau$.
\end{proposition}

\begin{proof}
The inequalities~(\ref{eq:Rtau-a}) 
and~(\ref{eq:Rtau-b}) show that $\mathbf{0}$ is 
the unique interior lattice point of $\rR_\tau$ 
and that its polar dual has vertices with integer
coordinates.  
\end{proof}

Let $\rR^\vee_\tau$ be the reflexive polytope 
which is polar dual to $\rR_\tau$. By definition, 
$\rR^\vee_\tau$ is the convex hull of the vectors 
$-{\rm \mathbf{e}}_e$ for all $e \in E$ and
$\sum_{e \in \iI} {\rm \mathbf{e}}_e$ for all
nonempty order ideals $\iI$ of $\tau$. Let us 
denote this set of vectors by $V_\tau$. Since 
$V_\tau$ is a subset of the set of vertices of 
the convex hull of the unit cube $[0,1]^n$ and 
the vectors $-{\rm \mathbf{e}}_e$ for all 
$e \in E$, every element of $V_\tau$ is a vertex 
of $\rR^\vee_\tau$ and $\rR^\vee_\tau$ has no 
lattice points other than $\mathbf{0}$ and its 
vertices.

We now state two conjectures about the 
$h^*$-vector of $\rR^\vee_\tau$. For the 
example of Figure~\ref{fig1} we have 
$h^\ast(\rR_\tau) = (1,228,1614,1614,218,1)$ 
and $h^\ast(\rR^\vee_\tau) = (1,7,16,16,7,1)$;
in particular, $\rR_\tau$ has 234 lattice 
points and $\rR^\vee_\tau$ has 13 lattice 
points.

\begin{conjecture} \label{conj:Rtau-a}
We have $h^\ast(\rR^\vee_\tau) = h^\ast
(\rR^\vee_{\tau^\ast})$ (equivalently, $\Ehr
(\rR^\vee_\tau, t) = \Ehr(\rR^\vee_{\tau^\ast}, 
t)$) for every preorder $\tau$.
\end{conjecture}

\begin{conjecture} \label{conj:Rtau-b}
We have $h^\ast(\rR^\vee_\tau,t) = h^\ast
(\rR^\vee_{\tau_1},t) \, h^\ast(\rR^\vee_{\tau_2},t)$ 
if the preorder $\tau$ is the ordinal sum of 
$\tau_1$ and $\tau_2$.
\end{conjecture}

We are unable to explain the latter conjecture 
in terms of the structure of $\rR^\vee_\tau$ 
since, in general, this polytope cannot be 
expressed as a free product.

\begin{remark} \rm
Curiously, the roots of the Ehrhart polynomial 
of $\rR^\vee_\tau$ lie on the line 
$\operatorname{Re}(z)=-1/2$ for every preorder
$\tau$ of size at most 6. Assuming the validity
of Conjecture~\ref{conj:Rtau-b}, this property 
fails when $\tau$ is the ordinal sum of two 
5-element antichains.
\end{remark}

%
%

\section{Double Ehrhart theory}
\label{sec:double}

This section is concerned with the lattice point 
enumeration of the polytope $\qQ_\tau(r,s)$, 
introduced in Section~\ref{sec:Minkowski}.

\begin{proposition}
There exists a polynomial $\biE_\tau(u,v)$ of 
total degree $n$ such that $\biE_\tau(r,s)$ is 
equal to the number of lattice points of 
$\qQ_\tau(r,s)$ for all $r,s \in \NN$.
\end{proposition}

\begin{proof}
Since $\qQ_\tau(r,s)$ can be expressed as the 
Minkowski sum of two $n$-dimensional lattice 
polytopes with dilation factors $r$ and $s$ by
Equation~(\ref{eq:def-Qrs}), the result follows 
from a general theorem of Bernstein--McMullen 
which defines multivariate Ehrhart polynomials; 
see, for instance, \cite[Theorem 19.4]{gruber}. 
Alternatively, the explicit formula 
\begin{equation} \label{eq:Ehr-double}
\biE_\tau(r,s) = \sum_{(a_e) \in \gG_{G_\tau}} 
\binom{s + a_0}{a_0} \prod_{e \in E} 
\binom{r + a_e - 1}{a_e} 
\end{equation}
for the number of lattice points of $\qQ_\tau
(r,s)$ follows from Theorem~\ref{thm:post} for
$t=1$ and the parameters $y_0 = s$ and $y_e = r$
for $e \in E$, where the sum ranges over all 
$G_\tau$-draconian sequences 
$(a_e) \in \NN^{E \sqcup \{0\}}$ characterized 
in Lemma~\ref{lem:tau-drac}. 
\end{proof}

For our running example from Figure~\ref{fig1}
we have
\begin{multline*}
\biE_\tau(u,v) = (760 u^5 + 1680 u^4 v + 990 
u^3 v^2 + 230 u^2 v^3 + 25 uv^4 + v^5 + 2740 
u^4 + 4350 u^3 v \\ + 1830 u^2 v^2 + 280 uv^3 
+ 15 v^4 + 3860 u^3 + 4180 u^2 v + 1115 uv^2 
\\ + 85 v^3 + 2660 u^2 + 1760 uv + 225 v^2 + 
900 u + 274 v + 120) / 120
\end{multline*}
and, in particular,
\[ \begin{aligned}
   \Ehr(\qQ_\tau, t) &= (760 t^5 + 2740 t^4 + 
   3860 t^3 + 2660 t^2 + 900 t + 120)/120 \\
   \Ehr(\qQ_\tau(1,1), t) &= (3686 t^5 + 9215
	 t^4 + 9240 t^3 + 4645 t^2 + 1174 t + 120)/
	 120. \end{aligned} \]
The terms of total degree $n=5$ of 
$\biE_\tau(u,v)$ yield the decomposition $1 + 
25 + 230 + 990 + 1680 + 760 = 3686$ of the 
normalized volume of the reflexive polytope 
$Q_\tau(1,1)$.

Since $\biE_\tau(t,t)$ is the Ehrhart polynomial 
of the reflexive polytope $\qQ_{\tau}(1,1)$, we 
have 
\[ \biE_\tau(-u,-u) = (-1)^n \,
   \biE_\tau(u-1,u-1) \]
for every preorder $\tau$ of size $n$ by Ehrhart 
reciprocity. The following stronger statement has 
been verified by computer for all preorders of 
size at most 5.
\begin{conjecture} \label{conj:double}
The double Ehrhart polynomial $\biE_\tau(u,v)$ 
satisfies the identity
\begin{equation} \label{eq:conj-double}
\biE_\tau(-u,-v) = (-1)^n \, \biE_\tau(u-1,v-1) 
\end{equation}
for every preorder $\tau$ of size $n$.
\end{conjecture}

\begin{remark} \label{rem:equivalence} \rm
The specialization $v=0$ of 
Conjecture~\ref{conj:double} is in fact 
equivalent to Conjecture~\ref{conj:h-ast}.
Indeed, let us use the notation of 
Section~\ref{sec:Ehrhart}. Setting $u=-m$ and 
$v=0$ in~(\ref{eq:conj-double}) and applying 
Equation~(\ref{eq:Ehr-double}) we get 

\[ \begin{aligned}
   \Ehr(\qQ_\tau, m) &= \biE_\tau(-m-1,-1) \\
   &= (-1)^n \sum_{(a_e) \in \gG_{G_\tau}} 
	    \binom{a_0-1}{a_0} \prod_{e \in E} 
	    \binom{-m + a_e - 2}{a_e},
   \end{aligned} \]
where the sum ranges over all $G_\tau$-draconian 
sequences $(a_e) \in \NN^{E \sqcup \{0\}}$ 
characterized in Lemma~\ref{lem:tau-drac}. Since
the summand is nonzero only for $a_0 = 0$, we 
reach the formula 
\[ \Ehr(\qQ_\tau, m) = \sum_{(a_e)} \prod_{e \in E} 
	    \binom{m + 1}{a_e}, \]
where the sum ranges over all points $\mathbf{a} = 
(a_e) \in P_{\tau^\ast}$ such that $|\mathbf{a}|=n$
(these are the maximal elements of $P_{\tau^\ast}$).  
Since (see, for instance, \cite[Section~4]{AXY26})
\[ \sum_{m \ge 0} \left( \, \prod_{e \in E} 
	 \binom{m+1}{a_e} \right) t^m = 
	 \frac{\displaystyle\sum_{w \in 
	 \mathfrak{S}_{\mathbf{a}}}
	  t^{n-1-\des(w)}}{(1-t)^{n+1}} \]
for every such $\mathbf{a} \in P_{\tau^\ast}$, 
our formula for $\Ehr(\qQ_\tau, m)$ is equivalent 
to Conjecture~\ref{conj:h-ast}.
\qed
\end{remark}

\begin{remark} \rm
There are nonisomorphic preorders with equal
double Ehrhart polynomials. The smallest pair 
for which this coincidence occurs is
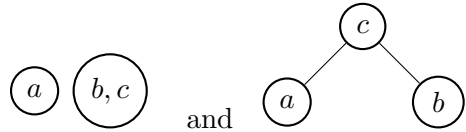
\begin{figure}[!h]
\centering
\begin{tikzpicture}
\tikzstyle{every node}=[draw,shape=circle, thick]

\node at (0,0) (a) {$a$};
\node at (1,0) (bc) {$b,c$};
\end{tikzpicture}
\quad\text{and}\quad
\begin{tikzpicture}
\tikzstyle{every node}=[draw,shape=circle, thick]

\node at (0,0) (a) {$a$};
\node at (2,0) (b) {$b$};
\node at (1,1) (c) {$c$};

\path [draw] (a) -- (c);
\path [draw] (b) -- (c);
\end{tikzpicture}.
\end{figure}

\end{remark}

\section{The $M$-triangle}
\label{sec:M-triangles}

This section discusses another enumerative 
invariant of the poset of lattice points of a 
preorder polytope, namely the $M$-triangle, and
extends in this setting a conjecture of 
\cite{Cha25b} on arbor polytopes. This 
conjecture provided much of the motivation 
behind the study of arbor polytopes; see 
\cite[Section~1.1]{Cha25b} for the global 
picture, which should extend verbatim to all 
preorder polytopes.

Let $P$ be a finite graded poset with rank 
function $\rho: P \to \NN$. The 
\emph{$M$-triangle} of $P$ is defined by the 
formula  
\[ M_P(x,y) = \sum_{a, b \, \in P : \, 
   a \preceq b} 
   \mu_P (a,b) x^{\rho(a)} y^{\rho(b)}, \]
where $\mu_P$ is the Möbius function of $P$.
The \emph{transmutation} of a rational 
function $f(x,y) \in \QQ(x,y)$ is defined as
\[ \overline{f}(x,y) = f \left( \frac{1-y}{1-xy}, 
   1-xy \right). \]

\smallskip
\noindent
We note that transmutation is an involution on 
$\QQ(x,y)$ and that $\overline{M}_P(x,y) \in 
\ZZ[x,y]$ for every graded poset $P$. The following 
conjecture is stated implicitly for arbors in 
\cite[Section~1.1]{Cha25b}.
\begin{conjecture} \label{conj:lattice}
For every preorder $\tau$ of size $n$ there 
exists a finite graded lattice $L_\tau$ of
rank $n$ such that $\zZ(P_\tau,t) = \zZ(L_\tau,
t)$ and
\begin{equation} \label{eq:M-transmute}
\overline{M}_{P_\tau}(x,y) = M_{L_\tau}(x,y). 
\end{equation}
In particular, the rank-generating polynomial 
of $L_\tau$ is equal to $h(\tau,t)$.
\end{conjecture}

For example, if $\tau$ is the chain preorder
of size $n$ having $n$ vertices of cardinality 
1 or one vertex of cardinality $n$, then the 
role of $L_\tau$ can be played by the 
noncrossing partition lattice of type $\TA_n$ 
and $\TB_n$, respectively; see Sections~7 
and~8 of \cite{Cha25b}. Thus, the poset 
$P_\tau$ relates the combinatorics of the 
noncrossing partition lattice to that of the
associahedron of corresponding type, the 
$h$-polynomial of which is equal to $h(\tau,
t)$. Conjecture~\ref{conj:lattice} predicts 
a similar relationship in a much more general 
context; we refer the reader to 
\cite[Section~1.1]{Cha25b} for an extensive 
discussion of this conjectural relationship 
in the setting of arbors.

The fact that Conjecture~(\ref{conj:lattice})
implies that the rank-generating polynomial 
of $L_\tau$ is equal to $h(\tau,t)$ follows
as in the special case of arbors; see 
\cite[Proposition~5.1]{Cha25b}.

Our second conjecture states that the 
transmuted $M$-triangle of $P_\tau$ satisfies 
a certain duality.
\begin{conjecture} \label{conj:M-duality}
We have
\[ (xy)^n \, \overline{M}_{P_\tau}(1/x,1/y) 
   = \overline{M}_{P_\tau}(y,x) \]
for every preorder $\tau$ of size $n$. 

Equivalently, $M_{L_\tau}(x,y) = 
M_{(L_\tau)^\ast}(x,y)$, where $(L_\tau)^\ast$ 
is the lattice dual to $L_\tau$.
\end{conjecture}

\begin{example} \rm
Let us illustrate these concepts and 
conjectures with our running example.
The $M$-triangle can be displayed as a matrix 
of coefficients, with the constant term on the 
bottom left. For the running example from
Figure~\ref{fig1} in dimension 5 with 92 
elements and 18 maximal elements, this matrix
is
\begin{equation*}
\begin{array}{rrrrrr}
-1 & 11 & -43 & 75 & -60 & 18 \\
5 & -35 & 84 & -84 & 30 &  \\
-10 & 43 & -57 & 24 &  &  \\
10 & -24 & 14 &  &  &  \\
-5 & 5 &  &  &  &  \\
1 &  &  &  &  & 
\end{array}
\end{equation*}
where the diagonal counts elements according 
to their sum of coordinates. The leftmost column 
comes from the unit cube contained in $\qQ_\tau$.
The transmutation of this $M$-triangle is
displayed by the matrix
\begin{equation*}
\begin{array}{rrrrrr}
-18 & 60 & -75 & 43 & -11 & 1 \\
60 & -174 & 180 & -77 & 11 &  \\
-75 & 180 & -139 & 34 &  &  \\
43 & -77 & 34 &  &  &  \\
-11 & 11 &  &  &  &  \\
1 &  &  &  &  & 
\end{array}
\end{equation*}
where the diagonal is now the $h$-vector. The 
identity claimed in Conjecture~\ref{conj:M-duality} 
is represented by the symmetry of the transmuted 
matrix in the main diagonal. Note that the 
number of maximal elements becomes, up to sign, the 
corner entry of the transmuted matrix.
\end{example}

\section{Examples}
\label{sec:exa}

The context of preorders allows for several 
interesting new examples of preorder polytopes, 
compared to the more restrictive setting of arbors.
This section studies such examples in some detail.
We let 
\[ \nara^\TA_n(t) = \sum_{i=0}^n \frac{1}{i+1}
\binom{n}{i}\binom{n+1}{i} t^i \]
and 
\[ \nara^\TB_n(t) = \sum_{i=0}^n \binom{n}{i}^2 
   t^i \]
be the Narayana polynomials of type $\TA_n$ and 
$\TB_n$, respectively.

\subsection{Zig-zag posets}
\label{sec:zig-zag}

Given a positive integer $n$, the 
\emph{zig-zag poset} $Z_n$ is defined as the 
partial order on the set $\{1, 2,\dots,n\}$ 
with cover relations $i \prec i+1$ for all odd 
$i \in \{1, 2,\dots,n-1\}$ and $i \succ i+1$ 
for all even $i \in \{1, 2,\dots,n-1\}$. Let 
us denote by $\qQ_n$ the preorder polytope 
and by $h_n(t)$ the $h$-polynomial associated to 
$\tau = Z_n$. The $n$th \emph{Delannoy 
polynomial} \oeis{A008288}, denoted by $d_n(t)$, 
is defined by the recurrence 
\begin{equation} \label{eq:Delannoy}
d_n(t) = (1+t) d_{n-1}(t) + t d_{n-2}(t)    
\end{equation}
for $n \ge 2$, with initial conditions $d_0(t)
= 1$ and $d_1(t) = 1+t$; we refer the reader 
to \cite{WZC19} and the references given there 
for combinatorial interpretations and explicit
formulas for these polynomials. The sum of the 
coefficients of $d_n(t)$ satisfies the 
recurrence $d_n(1) = 2 d_{n-1}(1) + d_{n-2}(1)$ 
for $n \ge 2$, where $d_0(1) = 1$ and $d_1(1) = 
2$, which yields the Pell numbers \oeis{A000129}.
\begin{proposition} \label{prop:zig-zag}
We have $h_n(t) = d_n(t)$ for every $n \ge 1$. 
In particular, the number of lattice points of 
$\qQ_n$ is equal to the Pell number $d_n(1)$.
\end{proposition}

\begin{proof}
Let $n \ge 4$ be even. The set of lattice 
points $(a_1, a_2,\dots,a_n)$ of $\qQ_n$ can 
be expressed as the disjoint union of nonempty 
subsets by the following conditions: 
\begin{itemize}
\item $a_n=0$,
\item $a_n=1$,
\item $a_n=2$, which implies $a_{n-1}=0$.
\end{itemize}
As one can easily verify, these sets are in 
one-to-one correspondence with the set of 
lattice points of $\qQ_{n-1}$, $\qQ_{n-1}$ 
and $\qQ_{n-2}$, respectively, so that 
\[ h_n(t) = (1+t) h_{n-1}(t) + th_{n-2}(t). \]

Similarly, the set of lattice points $(a_1, 
a_2,\dots,a_{n+1})$ of $\qQ_{n+1}$ can be 
expressed as the disjoint union of nonempty 
subsets by the following conditions: 
\begin{itemize}
\item $a_{n+1} = 1$,
\item $a_{n+1} = 0$ and $a_n = 0$,
\item $a_{n+1} = 0$ and $a_n = 1$,
\item $a_{n+1} = 0$ and $a_n = 2$,
\item $a_{n+1} = 0$ and $a_n = 3$, which 
      implies $a_{n-1} = 0$.
\end{itemize}
These sets are in one-to-one correspondence 
with the set of lattice points of $\qQ_n$, 
$\qQ_{n-1}$, $\qQ_{n-1}$, $\qQ_{n-1}$ and 
$\qQ_{n-2}$, respectively, so  
\[ h_{n+1}(t) = t h_n(t) + (1+2t) h_{n-1}(t) 
   + t h_{n-2}(t) = (1+t) h_n(t) + t
	 h_{n-1}(t). \]
Setting $h_0(t) = 1$, and since $h_1(t) = 
1+t$, $h_2(t) = 1+3t+t^2$ and $h_3(t) = 
1+5t+5t^2+t^3$, we have shown that $h_n(t)$ 
satisfies the recurrence (\ref{eq:Delannoy}) 
for $n \ge 2$ and the proof follows.
\end{proof}
 
The Delannoy polynomial $d_n(t)$ is known to
have all properties stated in 
Conjectures~\ref{conj:hvector} - 
\ref{conj:roots}. The palindromicity of $d_n
(t)$ follows directly from~(\ref{eq:Delannoy}); 
its $\gamma$-positivity and real-rootedness 
was shown in \cite{WZC19}. The polynomial 
$d_n(t)$ was shown to be equal to the 
$h$-polynomial of a flag simplicial polytope 
(whose boundary complex is isomorphic to the 
independence complex of a planar ternary 
graph) in \cite{BDHKY26}. We suspect that 
$d_n(t)$ is also equal to the $h$-polynomial 
of the $n$-dimensional 
pellytope~\cite{pellytopes} \cite[\S 4]{HLRZ20};
see \cite[Conjecture~6.1.36]{almeter} for a 
related conjecture. 
This example is also related to the lattices 
of sashes~\cite{Law13,Law14}. 



\subsection{Ordinal sums of antichains}
\label{sec:antichains}

For positive integers $m,n$ we will denote by 
$\tau_{m,n}$ the set $\{1, 2,\dots,m\} \sqcup
\{1^\prime, 2^\prime,\dots,n^\prime\}$ partially
ordered so that the cover relations are $i \prec 
j^\prime$ for all $i,j$. Thus, $\tau_{m,n}$ is 
the ordinal sum of an $m$-element antichain and 
an $n$-element antichain. Let us denote by 
$\qQ_{m,n}$ the preorder polytope and by 
$h_{m,n}(t)$ the $h$-polynomial associated to 
$\tau_{m,n}$. 
\begin{proposition} \label{prop:antichains}
The number of lattice points of $\qQ_{m,n}$ is 
equal to 
\[ \sum_{k=0}^m \sum_{\ell=0}^n 
   \binom{k+\ell}{k} \binom{m}{k} \binom{n}{\ell} 
	 \]
and
\[ \begin{aligned}
   h_{m,n}(t) &= \sum_{k=0}^m \sum_{\ell=0}^n 
   \binom{k+\ell}{k} \binom{m}{k} \binom{n}{\ell} 
	 t^{m-k+\ell} \\
   &= \sum_{i \ge 0} \binom{m}{i} \binom{n}{i}
	    t^i (1+t)^{m+n-2i}
   \end{aligned} \]
for all $m, n$. In particular, $h_{m,n}(t)$
is palindromic and $\gamma$-positive. Moreover, 
all roots of $h_{m,n}(t)$ are real.
\end{proposition}

\begin{proof}
Clearly, it suffices to prove the proposed 
formulas for $h_{m,n}(t)$. By definition, the 
lattice points of $\qQ_{m,n}$ are the tuples 
$(a_1, a_2,\dots,a_m, b_1, b_2,\dots,b_n) \in
\NN^{m+n}$ such that $a_i \in \{0, 1\}$ for all 
$1 \le i \le n$ and 
\begin{equation} \label{eq:ab}
b_{j_1} + b_{j_2} + \cdots + b_{j_r} + a_1 + 
a_2 + \cdots + a_m \le m+r 
\end{equation}
for all $1 \le j_1 < j_2 < \cdots < j_r \le n$.
Suppose that we have $a_i = 1$ for $k$ indices 
$i \in \{1, 2,\dots,m\}$ and $b_j \ge 1$ for 
$\ell$ indices $j \in \{1, 2,\dots,n\}$ and 
set $c_j = b_j - 1$ for such indices $j$. By 
symmetry, we may assume that $b_1, 
b_2,\dots,b_\ell \ge 1$. Then, the inequalities
(\ref{eq:ab}) become
\[ c_{j_1} + c_{j_2} + \cdots + c_{j_r} \le 
   m-k \]
for $1 \le j_1 < j_2 < \cdots < j_r \le \ell$
and are thus equivalent to the single 
inequality
\[ c_1 + c_2 + \cdots + c_\ell \le m-k. \]
As a result, there are $\binom{m}{k}$ ways 
to choose $a_1, a_2,\dots,a_m \in \{0,1\}$ and 
for every such choice there are 
$\binom{n}{\ell}\binom{m-k+\ell}{\ell}$ ways 
to choose $b_1, b_2,\dots,b_n$. Therefore,
\[ h_{m,n}(t) = \sum_{k=0}^m \sum_{\ell=0}^n 
   \binom{m-k+\ell}{\ell} \binom{m}{k} 
	 \binom{n}{\ell} t^{k+\ell}. \]
This formula is equivalent to the first formula 
proposed for $h_{m,n}(t)$. Since the right-hand 
side is equal to the rank-generating polynomial 
of the poset of shuffles of $(1, 2,\dots,m)$ 
and $(1^\prime, 2^\prime,\dots,n^\prime)$
\cite{Gr88}, the second formula proposed for
$h_{m,n}(t)$ follows from 
\cite[Corollary~4.8]{Gr88}.

Finally, we observe that $\sum_{i \ge 0} 
\binom{m}{i} \binom{n}{i} t^i$ is real-rooted as
a Hadamard product of real-rooted polynomials
\cite[V~155]{PS72}. By \cite[Remark~3.1.1]{Ga05},
so is $h_{m,n}(t)$.
\end{proof}

The enumerative combinatorics of the shuffle 
poset \cite{Gr88} is closely related to that 
of the bubble lattice for the same parameters, 
introduced and studied in~\cite{MM22,MM24,MM25}.
A shellable triangulation of a sphere with 
$h$-polynomial equal to $h_{m,n}(t)$ is the 
noncrossing bipartite complex of 
\cite[Section~4]{MM25}; see also
\cite[Remark~4.16]{MM25}.


\subsection{Grid posets}
\label{sec:grid}

Let us consider the direct product $C_n \times 
C_2$ of an $n$-element chain and a two-element
chain, for positive integers $n$. For small 
values of $n$, the number of lattice points of
the associated preorder polytope and the 
$h$-vector are as follows:

\bigskip
\begin{center}
\begin{tabular}{c c c}
  \toprule
  $n$ & number of lattice points & $h$-vector\\
  \midrule
  1 & 5 & $(1, 3, 1)$ \\
  2 & 38 & $(1, 9, 18, 9, 1)$ \\
  3 & 352 & $(1, 18, 86, 142, 86, 18, 1)$ \\
  4 & 3659 & $(1, 30, 260, 882, 1313, 882, 260, 
	30, 1)$ \\
  \bottomrule
\end{tabular}
\end{center}

\bigskip
\noindent
This suggests that the number of lattice points 
is given by sequence \oeis{A158266}. For the 
poset obtained from $C_n \times C_2$ by removing
the minimum and maximum elements we have

\bigskip
\begin{center}
\begin{tabular}{c c c}
  \toprule
  $n$ & number of lattice points & $h$-vector\\
  \midrule
  1 & 4 & $(1, 2, 1)$ \\
  2 & 29 & $(1, 7, 13, 7, 1)$ \\
  3 & 265 & $(1, 15, 65, 103, 65, 15, 1)$ \\
  4 & 2745 & 
	$(1, 26, 206, 659, 961, 659, 206, 26, 1)$ \\
  \bottomrule
\end{tabular}
\end{center}

\bigskip
\noindent
which suggests that the number of lattice points 
is given by sequence \oeis{A370955}. Curiously, 
the generating functions 
\[ 1 + x + 5 x^2 + 38 x^3 + 352 x^4 + 3659 x^5 
     + \cdots \]
and
\[ 1 - x - 4 x^2 - 29 x^3 - 265 x^4 - 2745 x^5 
     - \cdots \]
of these sequences are inverses of each other. 
This seems to extend to the generating 
functions of the $h$-polynomials as follows. 
Setting 
\[ \begin{aligned}
F(t,x) &= 1 + x + \left(t + 3 + t^{-1} \right) 
          x^2 + \left(t^2 + 9t + 18 + 9t^{-1} + 
					t^{-2} \right) x^3 + \cdots \\
G(t,x) &= 1 - x - \left(t + 2 + t^{-1} \right) 
              x^2 - \left(t^2 + 7t + 13 + 7t^{-1} 
							+ t^{-2} \right) x^3 - \cdots
  \end{aligned} \]
we suspect that $F(t,x) = 1/G(t,x) = 
\exp(S(t,x))$, where 
\begin{equation*}
s_n(t) = \sum_{i=0}^{n-1} \binom{n}{i}
\binom{n - 1}{i} t^i \quad \text{and} \quad
S(t,x) = \sum_{n \geq 1} s_n(t) s_n(t^{-1}) x^n/n.
\end{equation*}

For the poset obtained from $C_n \times C_2$ by 
removing the maximum element we have

\bigskip
\begin{center}
\begin{tabular}{c c c}
  \toprule
  $n$ & number of lattice points & $h$-vector\\
  \midrule
  1 & 2 & $(1, 1)$ \\
  2 & 12 & $(1, 5, 5, 1)$ \\
  3 & 100 & $(1, 12, 37, 37, 12, 1)$ \\
  4 & 980 & 
	$(1, 22, 138, 329, 329, 138, 22, 1)$ \\
  \bottomrule
\end{tabular}
\end{center}

\bigskip
\noindent
which suggests that the number of lattice 
points of the associated preorder polytope is 
given by sequence \oeis{A000888} and that the 
$h$-polynomial is equal to $\nara^{\TA}_{n-1}
(t) \nara^{\TB}_n(t)$.

\subsection{Double chains}
\label{sec:double-chains}

Let us consider the chain preorder having $n$ 
vertices, each of cardinality 2. For small 
values of $n$, we have:

\bigskip 
\begin{center}
\begin{tabular}{c c c}
  \toprule
  $n$ & number of lattice points & $h$-vector\\
  \midrule
  1 & 6 & $(1, 4, 1)$ \\
  2 & 53 & $(1, 12, 27, 12, 1)$ \\
  3 & 554 & $(1, 24, 134, 236, 134, 24, 1)$ \\
  4 & 6362 & $(1, 40, 410, 1540, 2380, 1540, 
	410, 40, 1)$ \\
  \bottomrule
\end{tabular}
\end{center}

\bigskip
\noindent
The lattice path interpretation of the lattice 
points described (for all chain preorders) in 
\cite[Section~2]{Ath26} shows that the number of 
lattice points is given by sequence 
\oeis{A066357}. The generating function 
\[ 1 + x + 6 x^2 + 53 x^3 + 554 x^4 + 6362 x^5 + 
   \cdots \]
of the number of lattice points is the inverse of
\[ 1 - \sum_{n \ge 1} \frac{1}{2n} \binom{4n-2}{2n-1} 
   x^n = 1 - x - 5 x^2 - 42 x^3 - 429 x^4 - 4862 x^5 
	 - \cdots \]
and the generating function
\[ 1 + x + \left(t + 4 + t^{-1} \right) x^2 + 
   \left(t^2 + 12t + 27 + 12t^{-1} + t^{-2} 
	 \right) x^3 + \cdots \]
of the $h$-polynomial is the inverse of $1 - 
\sum_{n \ge 1} \nara^\TA_{2n-2}(t) \, t^{-n+1} 
x^n$.

\subsection{Multiple chains}
\label{sec:multiple}

Let us consider the chain preorder $\tau_{n,k}$ 
having $n$ vertices, each of cardinality $k$. 
Experimental evidence suggests that 
\[ 1 + \sum_{n \ge 1} h(\tau_{n,k},t^2) \, t^{-nk} 
   x^n = \exp \left( \, \sum_{n \ge 1} 
	 \nara^\TB_{kn}(t^2) \, t^{-kn} \, \frac{x^n}{n} 
	 \right). \]
When $k$ is even, one can replace $t$ by $t^{1/2}$ 
and still get Laurent polynomials in $t$. For 
$k=2$ this gives a different interpretation of 
the $h$-polynomial from that of 
Section~\ref{sec:double-chains}; see \oeis{A181145}.


\subsection{Comb posets}
\label{sec:comb}

Let $\tau_n$ be the poset defined recursively 
for $n \in \NN$ as follows. The poset $\tau_0$ 
is empty and $\tau_{n+1}$ is obtained from $\tau_n$ 
by first taking the disjoint union with a poset
with one element and then adding a new minimum 
element. For small values of $n$, we have:

\bigskip
\begin{center}
\begin{tabular}{c c c}
  \toprule
  $n$ & number of lattice points & $h$-vector\\
  \midrule
  1 & 5 & $(1, 3, 1)$ \\
  2 & 33 & $(1, 8, 15, 8, 1)$ \\
  3 & 249 & $(1, 15, 61, 95, 61, 15, 1)$ \\
  4 & 2033 & $(1, 24, 166, 484, 683, 484, 166, 
	24, 1)$ \\
  \bottomrule
\end{tabular}
\end{center}

\bigskip
\noindent
which suggests that the number of lattice 
points of the associated preorder polytope 
is given by sequence \oeis{A034015}. The 
generating function
\[ x + 5 x^2 + 33 x^3 + 249 x^4 + 2033 x^5 + 
   \cdots \]
of this sequence is the compositional inverse 
of 
\[ \sum_{n \ge 1} (-1)^{n-1} \left( 1 + 2^n 
   (n-1) \right) x^n = x - 5 x^2 + 17 x^3 - 
	 49 x^4 + 129 x^5 - \cdots, \]
see \oeis{A000337}. This leads us to suspect 
that the generating function of the 
$h$-polynomial
\[ x + \left(t + 3 + t^{-1} \right) x^2 + 
   \left(t^2 + 8t + 15 + 8t^{-1} + t^{-2} 
	\right) x^3 + \cdots \]
is the compositional inverse of
\[ \sum_{n \ge 1} (-1)^{n-1} \left( 1 + 
   {(1+t)}^n (1+t+\cdots+t^{n-2}) \, t^{-n+1} 
	 \right) x^n. \]

\end{document}